\documentclass[11pt]{amsart}
\usepackage{url,amsmath,amssymb,amsthm,amsthm,mathtools,enumitem,tikz,graphicx,eucal,booktabs,array,comment}
\usetikzlibrary{calc}
\usepackage[top=30mm, bottom=35mm, left=30mm, right=30mm]{geometry}
\DeclareFontEncoding{OT2}{}{}
\DeclareFontSubstitution{OT2}{cmr}{m}{n}
\DeclareFontFamily{OT2}{cmr}{\hyphenchar\font45 }
\DeclareFontShape{OT2}{cmr}{m}{n}{
<5><6><7><8><9>gen*wncyr
<10><10.95><12><14.4><17.28><20.74><24.88>wncyr10}{}
\DeclareFontShape{OT2}{cmr}{b}{n}{
<5><6><7><8><9>gen*wncyb
<10><10.95><12><14.4><17.28><20.74><24.88>wncyb10}{}
\DeclareMathAlphabet{\mathcyr}{OT2}{cmr}{m}{n}
\DeclareMathAlphabet{\mathcyb}{OT2}{cmr}{b}{n}
\SetMathAlphabet{\mathcyr}{bold}{OT2}{cmr}{b}{n}

\usepackage{hyperref}
\usepackage{xcolor}
\usepackage[capitalize,nameinlink,noabbrev,nosort]{cleveref}
\hypersetup{
	colorlinks=true,       
	linkcolor=cyan,          
	citecolor=cyan,        
	filecolor=cyan,      
	urlcolor=cyan,           
}
\newtheorem{theoremcounter}{Theorem Counter}[section]
\theoremstyle{plain}
\newtheorem{theorem}[theoremcounter]{Theorem}
\newtheorem{proposition}[theoremcounter]{Proposition}
\newtheorem{conjecture}[theoremcounter]{Conjecture}
\newtheorem{problem}[theoremcounter]{Problem}
\newtheorem{lemma}[theoremcounter]{Lemma}
\newtheorem{corollary}[theoremcounter]{Corollary}
\theoremstyle{definition}
\newtheorem{definition}[theoremcounter]{Definition}
\theoremstyle{remark}
\newtheorem{remark}[theoremcounter]{Remark}
\newtheorem{example}[theoremcounter]{Example}
\newcommand{\bn}{\boldsymbol{n}}
\newcommand{\dia}{\diamondsuit}
\newcommand{\bc}{\boldsymbol{c}}
\newcommand{\bdelta}{\boldsymbol{\delta}}
\newcommand{\cD}{\mathcal{D}}
\newcommand{\fD}{\mathfrak{D}}
\newcommand{\cF}{\mathsf{F}}

\newcommand{\bk}{\boldsymbol{k}}
\newcommand{\bl}{\boldsymbol{l}}
\newcommand{\cH}{\mathcal{H}}
\newcommand{\cZ}{\mathcal{Z}}
\newcommand{\NN}{\mathbb{N}}
\newcommand{\II}{\mathbb{I}}
\newcommand{\QQ}{\mathbb{Q}}
\newcommand{\RR}{\mathbb{R}}
\newcommand{\ZZ}{\mathbb{Z}}
\newcommand{\dep}{\mathrm{dep}}
\newcommand{\rH}{\mathrm{H}}
\newcommand{\wt}{\mathrm{wt}}
\numberwithin{equation}{section}

\allowdisplaybreaks

\DeclareMathOperator{\spa}{span}
\address{Graduate School of Science and Engineering, Kagoshima University, 1-21-35 Korimoto, \newline Kagoshima, Kagoshima 890-0065, Japan}
\email{hirose@sci.kagoshima-u.ac.jp}
\address{Faculty of Mathematics, Kyushu University, Motooka 744, Nishi-ku, Fukuoka, 819-0395, Japan}
\email{nozaki.takumi.912@s.kyushu-u.ac.jp}
\address{Nagahama Institute of Bio-Science and Technology, 1266, Tamura, Nagahama, Shiga, 526-0829, Japan}
\email{s\_seki@nagahama-i-bio.ac.jp}
\address{College of Arts and Sciences, University of Tokyo, 3-8-1 Komaba, Meguro-ku, Tokyo 153-8914, Japan}
\email{taiki-watanabe@g.ecc.u-tokyo.ac.jp}

\thanks{This research was supported by JSPS KAKENHI Grant Numbers JP22K03244 (Hirose) and JP21K13762 (Seki).}
\dedicatory{Dedicated to Professor Masanobu Kaneko on the occasion of his 60+4\textsuperscript{th} birthday}

\title[]{The $\ZZ$-module of multiple zeta values is generated by ones for indices without ones}
\author[]{Minoru Hirose, Takumi Maesaka, Shin-ichiro Seki, and Taiki Watanabe}
\date{}

\begin{document}

\begin{abstract}
We prove that every multiple zeta value is a $\ZZ$-linear combination of $\zeta(k_1,\dots, k_r)$ where $k_i\geq 2$.
Our proof also yields an explicit algorithm for such an expansion.
The key ingredient is to introduce modified multiple harmonic sums that partially satisfy the relations among multiple zeta values and to determine the structure of the space generated by them.
\end{abstract}

\maketitle

\section{Introduction}\label{sec:intro}
\subsection{Previous results and an open problem}
The \emph{multiple zeta value} $\zeta(\bk)$, associated with a tuple of positive integers $\bk=(k_1,\dots, k_r)$ satisfying $k_r\geq 2$ (called an \emph{admissible index}), is defined by
\[
\zeta(\bk)\coloneqq\sum_{0<n_1<\cdots<n_r}\frac{1}{n_1^{k_1}\cdots n_r^{k_r}}.
\]
We call $\wt(\bk)\coloneqq k_1+\cdots+k_r$ the \emph{weight} of an admissible index $\bk$.
For an integer $k\geq 2$, the set of all admissible indices of weight $k$ is denoted by $\mathbb{I}^{\mathrm{adm}}_k$.
Let $\cZ_k=\spa_{\QQ}\{\zeta(\bk) \mid \bk\in\II^{\mathrm{adm}}_k\}$ denote the $\QQ$-vector space spanned by all multiple zeta values of weight $k$.
The space $\cZ_k$ has been extensively studied by many researchers in order to uncover its rich structure.
The dimension $\dim_{\QQ}\cZ_k$ is significantly smaller than $\#\II^{\mathrm{adm}}_k = 2^{k-2}$, and the following was conjectured by Zagier.
Here, for a finite set $S$, let $\#S$ denote the cardinality of $S$.
Let $(d_k)_{k\geq 0}$ be the sequence of integers defined by $d_0=1$, $d_1=0$, $d_2=1$ and the recursion $d_k=d_{k-2}+d_{k-3}$.
Asymptotically, we have $d_k\sim 0.4114\ldots\times(1.3247\ldots)^k$ as $k\to\infty$.
\begin{conjecture}[Zagier~\cite{Zagier1994}]\label{conj:dimension}
For every integer $k\geq 2$, we have
\[
\dim_{\QQ}\cZ_k = d_k.
\]
\end{conjecture}
The lower bound estimate $\dim_{\QQ}\cZ_k\geq d_k$ involves transcendental difficulties and remains out of reach.
However, the upper bound estimate was independently proved by Deligne--Goncharov and Terasoma using the theory of mixed Tate motives.
\begin{theorem}[Deligne--Goncharov~\cite{DeligneGoncharov2005}, Terasoma~\cite{Terasoma2002}]\label{thm:DGT}
For every integer $k\geq 2$, we have
\[
\dim_{\QQ}\cZ_k\leq d_k.
\]
\end{theorem}
Let $\II_k^{\rH}\coloneqq\{(k_1,\dots,k_r) \in \II^{\mathrm{adm}}_k \mid k_1,\dots, k_r\in\{2,3\}\}$.
It is easy to see that $\#\II_k^{\rH}=d_k$.
Hoffman~\cite{Hoffman1997} conjectured that $\cZ_k$ is generated by $\{\zeta(\bk) \mid \bk\in\II_k^{\rH}\}$, which is called the \emph{Hoffman basis}, and Brown later proved this conjecture using the theory of mixed Tate motives with Zagier's formula for $\zeta(2,\dots,2,3,2,\dots,2)$ (\cite{Zagier2012}).
\begin{theorem}[Brown~\cite{Brown2012}]\label{thm:Brown}
For every integer $k\geq 2$, we have
\[
\cZ_k=\spa_{\QQ}\{\zeta(\bk) \mid \bk\in\II_k^{\rH}\}.
\]
\end{theorem}
Brown's theorem is stronger than the inequality of Deligne--Goncharov and Terasoma and stands as a landmark result in the theory of multiple zeta values.
By Brown's theorem, any given multiple zeta value $\zeta(\bk)$ of weight $k$ can be expressed as
\[
\zeta(\bk)=\sum_{\bl\in\II_k^{\rH}}c^{\rH}_{\bk;\bl}\zeta(\bl)
\]
with rational coefficients $c^{\rH}_{\bk;\bl}$ (conjecturally, or motivically, uniquely). 
However, Brown’s proof only guarantees the existence of $c^{\rH}_{\bk;\bl}$ and does not provide an algorithm to compute them explicitly.
The following problems concerning $c^{\rH}_{\bk;\bl}$ are fundamental:
\begin{problem}\label{prob:fundamental}
\begin{enumerate}[font=\normalfont]
\item\label{it:algorithm} Find an algorithm to determine $c^{\rH}_{\bk;\bl}$ explicitly.
\item\label{it:odd-conjecture} Resolve the conjecture in \cite[Remark~15]{HiroseSato-pre2}$:$ the denominator of $c^{\rH}_{\bk;\bl}$ is always odd.
\end{enumerate}
\end{problem}
Analogues of \cref{prob:fundamental} have already been settled, for example, in two other settings: \emph{Euler sums} (also known as \emph{alternating multiple zeta values}) and \emph{multiple zeta values in positive characteristic}.
\begin{itemize}[leftmargin=1.5em]
\item Euler sums: Deligne \cite{Deligne2010}, using the theory of mixed Tate motives, proved, before Brown, a theorem corresponding to Brown’s result in the setting of Euler sums.
Later, Hirose--Sato \cite{HiroseSato-pre1} considered expansions with respect to a basis different from Deligne’s, gave an explicit algorithm to compute the corresponding expansion coefficients, and proved that these coefficients always lie in $\ZZ_{(2)}$.
Their proof is based on the confluence relations of level two and does not rely on Deligne's argument.
\item Multiple zeta values in positive characteristic: Todd \cite{Todd2018} formulated a dimension conjecture corresponding to Zagier’s conjecture (\cref{conj:dimension}), and Thakur \cite{Thakur2017} proposed a basis (called \emph{Thakur’s basis}) that is the counterpart of the Hoffman basis as a conjecture.
Ngo Dac \cite{NgoDac2021} subsequently proved the analogue of \cref{prob:fundamental} for expansions with respect to Thakur’s basis.
In this setting no denominators arise; every expansion coefficient belongs to $\mathbb{F}_q[\theta]$.
His proof is based on binary relations for finite sums.
In the positive characteristic case, transcendence techniques are well developed, and Todd's dimension conjecture has been independently settled by Chang--Chen--Mishiba~\cite{ChangChenMishiba2023} and Im--Kim--Le--Ngo Dac--Pham~\cite{Im-etal-2024}.
\end{itemize}
By combining Hirose--Sato’s algorithm for Euler sums with Brown’s result, part \eqref{it:algorithm} of \cref{prob:fundamental} has already been answered affirmatively (\cite[Section~1.3]{HiroseSato-pre1}, see also \cite{Brown2011}).
However, in view of the successes obtained for Euler sums and for multiple zeta values in positive characteristic, it remains an important goal to find an elementary proof in the classical setting that avoids both the theory of mixed Tate motives and Borel’s calculation of $K$-groups and that would settle the remaining part \eqref{it:odd-conjecture} of \cref{prob:fundamental}, which is still open.
\subsection{The first main result}\label{subsec:1st-main}
Let $\II_k^{\geq 2}\coloneqq\{(k_1,\dots,k_r)\in\II^{\mathrm{adm}}_k \mid k_1,\dots, k_r\geq 2\}$, which is a superset of $\II_k^{\rH}$.
In this paper, we consider expansions of multiple zeta values in terms of elements of $\{\zeta(\bk) \mid \bk \in \II_k^{\geq 2}\}$ instead of Hoffman's basis.
First, we recall the following result.
A tuple of positive integers $\bk=(k_1,\dots,k_r)$ that allows $k_r=1$ is called an \emph{index}.
We define the \emph{weight} and \emph{depth} of $\bk$ as $\wt(\bk)\coloneqq k_1+\cdots +k_r$ and $\dep(\bk)\coloneqq r$, respectively.
For positive integers $k$ and $r$ with $r\leq k$, the set of all indices of weight $k$ and depth $r$ is denoted by $\II_{k,r}$.
The notation $\{1\}^m$ denotes the string $1,\dots,1$ in which the entry $1$ is repeated $m$ times.
\begin{theorem}[Kaneko--Sakata~\cite{KanekoSakata2016}]\label{thm:KanekoSakata}
Let $a$ and $b$ be positive integers.
Then, we have
\[
\zeta(\{1\}^{a-1},b+1)=\sum_{r=1}^{\min\{a,b\}}(-1)^{r-1}\sum_{\substack{(c_1,\dots,c_r)\in \II_{a,r} \\ (d_1,\dots, d_r)\in \II_{b,r}}}\zeta(c_1+d_1,\dots, c_r+d_r).
\]
\end{theorem}
Since $(c_1+d_1,\dots, c_r+d_r)\in\II^{\geq 2}_{a+b}$, Kaneko--Sakata's theorem provides a formula for expanding multiple zeta values of weight $k$ and height~1 in terms of elements of $\II^{\geq 2}_k$.
Here, the number of components greater than 1 is called the \emph{height} of an index.
In general, it is clear from \cref{thm:Brown} that a multiple zeta value of weight $k$ can be expanded in terms of the elements of $\II^{\geq 2}_k$, but what is remarkable about Kaneko--Sakata's theorem is that the coefficients of the expansion are explicitly determined and turn out to be integers.
This result naturally raises the question of whether Kaneko--Sakata's theorem can be extended to general multiple zeta values.
Murahara and Sakata provided the following generalization.
For an index $\bk=(k_1,\dots,k_r)$, we write $\bl\preceq\bk$ if the index $\bl$ is obtained by inserting either a comma ``,'' or a plus ``+'' into each $\square$ in $(k_1 \, \square\cdots\square \, k_r)$.
\begin{theorem}[Murahara--Sakata~\cite{MuraharaSakata2018}]\label{thm:MuraharaSakata}
Let $a$, $b$, and $s$ be positive integers satisfying $a\geq s$, and let $(b_1,\dots, b_s)$ an index in $\II_{b,s}$.
Then, we have
\begin{align*}
&\sum_{(a_1,\dots,a_s)\in\II_{a,s}}\zeta(\{1\}^{a_1-1},b_1+1,\dots,\{1\}^{a_s-1},b_s+1)\\
&=\sum_{r=s}^{\min\{a,b\}}(-1)^{r-s}\sum_{\substack{(c_1,\dots,c_r)\in\II_{a,r} \\ (d_1,\dots,d_r)\succeq(b_1,\dots,b_s)}}\zeta(c_1+d_1,\dots,c_r+d_r).
\end{align*}
\end{theorem}
The case $s=1$ corresponds to \cref{thm:KanekoSakata}.
As can be seen from the proof of \cite[Section~3]{MuraharaSakata2018}, Murahara--Sakata's relations are included in the derivation relations (\cite{IharaKanekoZagier2006}).
See \cref{fig:relations}.
\cref{thm:MuraharaSakata} is a natural extension of \cref{thm:KanekoSakata}, but since the left-hand side involves a sum of multiple zeta values, \cref{thm:MuraharaSakata} does not provide a formula for expressing an individual multiple zeta value in terms of elements of $\II^{\geq 2}_k$.
In this paper, we extend Kaneko--Sakata's theorem to individual multiple zeta values without taking sums, and the following is the first main result.
\begin{theorem}\label{thm:main1}
Let $k\geq 2$ be an integer.
For every $\bk \in \II^{\mathrm{adm}}_k$, we have
\[
\zeta(\bk)=\sum_{\bl\in\II^{\geq 2}_k}c_{\bk;\bl}\zeta(\bl),
\]
where $c_{\bk;\bl}$ are certain explicitly given integers.
In particular, we have
\[
\spa_{\ZZ}\{\zeta(\bk)\mid \bk\in\II^{\mathrm{adm}}_k\}=\spa_{\ZZ}\{\zeta(\bk)\mid \bk\in\II^{\geq 2}_k\}.
\]
\end{theorem}
Since, as explained in \cref{sec:outline}, our proof is elementary, this result serves as an intermediate step toward the goal stated in the previous subsection. Moreover, it was previously unknown that the coefficients in expansions with respect to $\{\zeta(\bk)\mid\bk\in\II^{\ge 2}_k\}$ can always be taken to be \emph{integers}.

In order to describe the explicit algorithm for determining $c_{\bk;\bl}$, we use the notation of the Hoffman algebra.
Let $\cH\coloneqq\QQ\langle x,y\rangle$ be the non-commutative polynomial ring over $\QQ$ in two variables $x$ and $y$.
The subrings $\cH^0\subset\cH^1\subset\cH$ are defined by $\cH^0\coloneqq\QQ+y\cH x$ and $\cH^1\coloneqq\QQ+y\cH$, respectively.
Furthermore, the subring $\cH^{\geq 2}$ of $\cH^0$ is defined by $\cH^{\geq 2}\coloneqq\QQ+yx\QQ\langle x,yx\rangle$.
A linear mapping $\cD\colon \cH^0\to \cH^{\geq 2}$ is defined by the following recurrence relation.
Here, for an even-depth index $\bc=(c_1,\dots, c_{2s})$, we set $\fD(\bc)\coloneqq \cD(y^{c_1}x^{c_2}\cdots y^{c_{2s-1}}x^{c_{2s}})$.
The relation is
\begin{align*}
\fD(\bc)&=\sum_{\substack{A\subset[2s]^1_{\bc} \\ A: \text{even-odd}}}\sum_{\substack{B\subset[2s]^{>1}_{\bc} \\ \#A+\#B\geq 1 \\ \{2r,2r+1\}\not\subset B \ (r\in[s-1])}}(-1)^{\#B-1}\fD(\bc_{(-A)}-\bdelta_B)x^{\#A+\#B}\\
&\quad +\sum_{\substack{A\subset[2s]^1_{\bc} \\ A: \text{even-odd}}}\sum_{\substack{B\subset[2s]^{>1}_{\bc} \\ \#A+\#B\geq 2 \\ \{2r,2r+1\}\not\subset B \ (r\in[s-1])}}(-1)^{\#B-1}\fD(\bc_{(-A)}-\bdelta_B)yx^{\#A+\#B-1}\\
&\quad +\sum_{\substack{A\subset[2s]^1_{\bc} \\ A: \text{odd-even}}}\sum_{\substack{B\subset[2s]^{>1}_{\bc} \\ \#A+\#B\geq 2 \\ \{2r-1,2r\}\not\subset B \ (r\in[s])}}(-1)^{\#B}\fD(\bc_{(-A)}-\bdelta_B)yx^{\#A+\#B-1}
\end{align*}
with the initial value $\fD(\varnothing)=\cD(1)=1$.
Here we explain some of the notation used.
For a non-negative integer $n$, let $[n]$ denote the set $\{1,2,\dots, n\}$ ($[0]=\varnothing$).
For an index $\bk=(k_1,\dots, k_r)$, let $[r]_{\bk}^1$ denote the set of positions at which $k_i=1$, that is, $[r]_{\bk}^1=\{i\in [r] \mid k_i=1\}$.
Likewise, let $[r]^{>1}_{\bk}$ denote the set of positions at which $k_i>1$, that is, $[r]^{>1}_{\bk}=\{i\in [r] \mid k_i>1\}$.
We say that a subset $A\subset [2s]$ is \emph{even-odd} if whenever an even integer $i$ belongs to $A$, then $i+1$ also belongs to $A$, and whenever an odd integer $i$ belongs to $A$, then $i-1$ also belongs to 
$A$.
We say that a subset $A\subset [2s]$ is \emph{odd-even} if whenever an odd integer $i$ belongs to $A$, then $i+1$ also belongs to $A$, and whenever an even integer $i$ belongs to $A$, then $i-1$ also belongs to 
$A$.
For an index $\bk=(k_1,\dots, k_r)$ and a subset $A\subset[r]$, when the elements of $[r]\setminus A$ are written in increasing order as $i_1, \dots, i_t$, we denote the index $(k_{i_1},\dots, k_{i_t})$ by $\bk_{(-A)}$.
When $A=\varnothing$, we have $\bk_{(-\varnothing)}=\bk$.
For a set $S$ and an element $s\in S$, we define
\[
\delta_{s\in S}\coloneqq\begin{cases} 1 & \text{ if } s\in S, \\ 0 & \text{ otherwise.}\end{cases}
\]
When considering the notation $\bk_{(-A)}=(k_{i_1},\dots, k_{i_t})$ above, for a subset $B\subset[r]$ we define $\bk_{(-A)}-\bdelta_B$ as $(k_{i_1}-\delta_{i_1\in B},\dots, k_{i_t}-\delta_{i_t\in B})$.

For a positive integer $k$, we define $e_k\coloneqq yx^{k-1}$ ($e_1=y$).
Let $Z\colon\cH^0\to\RR$ be the $\QQ$-linear mapping defined by $Z(e_{k_1}\cdots e_{k_r})=\zeta(\bk)$ for each admissible index $\bk=(k_1,\dots, k_r)$ and $Z(1)=1$.
We shall prove the following.
\begin{theorem}\label{thm:the-algorithm}
For positive integers $s$ and $c_1,\dots, c_{2s}$, we have
\[
\zeta(\{1\}^{c_1-1},c_2+1,\dots, \{1\}^{c_{2s-1}-1},c_{2s}+1)=Z(\fD(c_1,\dots, c_{2s})).
\]
\end{theorem}
This equality provides an explicit algorithm for determining the coefficients $c_{\bk;\bl}$.

Strictly speaking, \cref{thm:KanekoSakata} is stated as an explicit closed formula, whereas \cref{thm:main1} with \cref{thm:the-algorithm} is stated as a recurrence formula, so their formulations differ.
Nevertheless, by definition we have
\begin{align*}
\fD(a,b)&=\fD(a-1,b)x+\fD(a,b-1)x-\fD(a-1,b-1)x^2-\fD(a-1,b-1)yx \ \text{if } a>1, b>1,\\
\fD(1,b)&=\fD(1,b-1)x\quad \text{if } a=1, b>1,\\
\fD(a,1)&=\fD(a-1,1)x\quad \text{if } a>1, b=1,\\
\fD(1,1)&=yx\quad \text{if } a=b=1,
\end{align*}
and applying $Z$ to this relations yields the formulation of \cref{thm:KanekoSakata} by the recurrence formula.
In this sense, \cref{thm:main1} can be regarded as a natural generalization of \cref{thm:KanekoSakata}.
Furthermore, although we do not give a direct proof, the inclusions in \cref{fig:relations} indicate that the family of relations obtained by \cref{thm:MuraharaSakata} is contained in the family of relations obtained by \cref{thm:the-algorithm}.
\begin{example}
We present several examples of expansions computed using our algorithm.
\begin{align*}
    \zeta(3,1,4)&=\zeta(5,3)-\zeta(4,4)-\zeta(3,3,2)+\zeta(2,4,2)-2\zeta(3,2,3)-\zeta(2,3,3), \\
    \zeta(4,1,1,2)&= 2\zeta(4,4)+4\zeta(3,5)+2\zeta(3,3,2)+2\zeta(3,2,3)+2\zeta(2,2,4)+\zeta(2,2,2,2),\\
    \zeta(3,1,3,2) &= \zeta(4,5) - 2\zeta(3,6) + \zeta(4,3,2) + \zeta(4,2,3) - 3\zeta(3,3,3) - \zeta(3,2,4) \\
    &\quad + \zeta(2,3,4) - 2\zeta(2,2,5) - 2\zeta(3,2,2,2) + 2\zeta(2,3,2,2) \\
    &\quad - \zeta(2,2,3,2) - 3\zeta(2,2,2,3), \\
    \zeta(1,2,1,1,5) &= \zeta(7,3) - \zeta(6,4) - \zeta(5,5) - \zeta(4,6) - \zeta(3,5,2) - 3\zeta(4,3,3) \\
    &\quad - 5\zeta(3,4,3) - 2\zeta(2,5,3) + 2\zeta(4,2,4) - 2\zeta(3,3,4) \\
    &\quad + \zeta(3,2,5) + \zeta(2,2,6) - \zeta(4,2,2,2) + 2\zeta(3,3,2,2) \\
    &\quad + \zeta(2,4,2,2) + \zeta(2,3,3,2) + \zeta(2,3,2,3) + \zeta(2,2,3,3) \\
    &\quad - \zeta(2,2,2,4) + \zeta(2,2,2,2,2), \\
    \zeta(2,3,1,2,2) &= 2\zeta(2,3,2,3) + \zeta(2,2,2,2,2), \\
    \zeta(2,2,2,1,3) &= \zeta(4,2,2,2) - \zeta(3,3,2,2) - \zeta(2,3,3,2) - \zeta(2,2,3,3) - 4\zeta(2,2,2,2,2), \\
    \zeta(1,1,6,1,2) &= \zeta(8,3) + \zeta(7,4) - 2\zeta(5,6) - 9\zeta(4,7) - 6\zeta(3,8) + \zeta(7,2,2) \\
    &\quad - \zeta(6,3,2) - 2\zeta(5,4,2) - 3\zeta(4,5,2) + 2\zeta(3,6,2) \\
    &\quad - 4\zeta(5,3,3) - 6\zeta(4,4,3) - 4\zeta(3,5,3) - \zeta(2,6,3) \\
    &\quad - \zeta(5,2,4) - 5\zeta(4,3,4) - 3\zeta(3,4,4) - \zeta(2,5,4) \\
    &\quad - 3\zeta(4,2,5) - 2\zeta(3,3,5) + \zeta(2,4,5) - 2\zeta(3,2,6) \\
    &\quad - 2\zeta(2,3,6) - 3\zeta(5,2,2,2) - \zeta(4,3,2,2) + 4\zeta(3,4,2,2) \\
    &\quad - \zeta(2,5,2,2) - 2\zeta(4,2,3,2) + 5\zeta(3,3,3,2) + \zeta(2,4,3,2) \\
    &\quad + 3\zeta(3,2,4,2) + 3\zeta(2,2,5,2) - 3\zeta(4,2,2,3) + 3\zeta(3,3,2,3) \\
    &\quad + \zeta(3,2,3,3) + 2\zeta(2,2,4,3) + 2\zeta(3,2,2,4) - \zeta(2,3,2,4) \\
    &\quad + \zeta(2,2,3,4) + 4\zeta(2,2,2,5) + 6\zeta(3,2,2,2,2) + 3\zeta(2,3,2,2,2) \\
    &\quad + 4\zeta(2,2,3,2,2) + 5\zeta(2,2,2,3,2) + 6\zeta(2,2,2,2,3),\\
    \zeta(3,3,2,1,2) &= 3\zeta(3,3,3,2) + \zeta(3,2,2,2,2) + \zeta(2,2,2,3,2). \\
\end{align*}
\end{example}
The \emph{$t$-interpolated multiple zeta value} $\zeta^t(\bk)$, introduced by Yamamoto (\cite{Yamamoto2013-2}), is a one-variable polynomial in $t$ and is defined for an admissible index 
$\bk$ by
\[
\zeta^t(\bk)\coloneqq\sum_{\bl\preceq\bk}t^{\dep(\bk)-\dep(\bl)}\zeta(\bl)\in\RR[t].
\]
This polynomial interpolates between $\zeta^0(\bk)=\zeta(\bk)$ and $\zeta^1(\bk)\eqqcolon\zeta^{\star}(\bk)$.
Since
\[
\zeta(\bk)=\sum_{\bl\preceq\bk}(-t)^{\dep(\bk)-\dep(\bl)}\zeta^t(\bl)
\]
holds, we have the following corollary from \cref{thm:main1}.
Note that if $\bl\preceq\bk$ and $\bk\in\II^{\geq 2}_k$, then $\bl\in\II^{\geq 2}_k$.
\begin{corollary}\label{cor:star-interpolate}
For every integer $k\geq 2$, we have
\[
\spa_{\ZZ[t]}\{\zeta^t(\bk)\mid \bk\in\II_k^{\mathrm{adm}}\}=\spa_{\ZZ[t]}\{\zeta^t(\bk)\mid \bk\in\II_k^{\geq 2}\}.
\]
In particular, we have
\[
\spa_{\ZZ}\{\zeta^{\star}(\bk)\mid \II^{\mathrm{adm}}_k\}=\spa_{\ZZ}\{\zeta^{\star}(\bk)\mid \bk\in\II_k^{\geq 2}\}.
\]
\end{corollary}
\begin{remark}
The star analogue $\cZ_k=\spa_{\QQ}\{\zeta^{\star}(\bk) \mid \bk\in\II^{\rH}_k\}$ of Hoffman's conjecture was conjectured by Ihara--Kajikawa--Ohno--Okuda~\cite{IharaKajikawaOhnoOkuda2011}, and it has already been proved by Glanois \cite[Theorem~4.2]{Glanois-pre} with Linebarger--Zhao's identity \cite[Theorem~1.2]{LinebargerZhao2015}.
This does not follow directly from Brown's result.
In contrast, \cref{cor:star-interpolate} follows directly from \cref{thm:main1}.
To the best of the authors’ knowledge, the dimension and generating sets of the space generated by $t$-interpolated multiple zeta values have so far been little studied.
\end{remark}
Let $F_n$ be the $n$th Fibonacci number.
Since $\#\II^{\geq 2}_k=F_{k-1}$ holds, we obtain the following corollary from \cref{thm:main1}.
\begin{corollary}
For every integer $k\geq 2$, we have
\[
\dim_{\QQ}\cZ_k\leq F_{k-1}.
\]
\end{corollary}
Of course, $F_{k-1}\sim 0.2763\ldots\times(1.6180\ldots)^k$ is greater than $d_k$, so this is a weaker result than \cref{thm:DGT}.
However, our proof is elementary and reduces the dimension using a concrete family of relations, and we believe it is worth mentioning.
\begin{remark}\label{rem:qMZV}
Bachmann conjectured, based on numerical experiments, that the space generated by Bradley--Zhao $q$-multiple zeta values is spanned by ones for indices whose components are all at least $2$.
This conjecture is mentioned in the paper by Brindle \cite[p.29]{Brindle2024}, written under the supervision of Bachmann.
It is an interesting problem whether one can resolve Bachmann’s conjecture by developing the methods presented in the present paper.
\end{remark}
\begin{remark}\label{rem:Ecalle}
After we released the first version of our manuscript, Hidekazu Furusho kindly informed us that Ecalle has also developed an algorithm for expanding multiple zeta values in terms of the set $\{\zeta(\bk) \mid \bk \in \II_k^{\geq 2}\}$.
He refers to this procedure as ``unit-cleansing''; see, for example, \cite[11.6]{Ecalle2011}.
Although his algorithm covers colored multiple zeta values, he notes that the expansion coefficients are rational numbers, and he does not address integrality.
While his proof is built on his mould theory, our approach is based on modified multiple harmonic sums called $\zeta^{\dia}$-values, and the two methods are completely different.
\end{remark}
\subsection{The second main result}
We define the space of relations obtained by \cref{thm:the-algorithm} as
\[
\mathsf{Drop1}\coloneqq\spa_{\QQ}\{w-\cD(w)\mid w\in y\cH x\}\subset\mathrm{Ker}(Z).
\]
We investigate its connections to other well-known families of relations.
Let $\phi$ be the automorphism on $\cH$ defined by $\phi(x)=x+y$ and $\phi(y)=-y$.
The harmonic (or stuffle) product $*$ on $\cH^1$ is defined by $w*1=1*w=w$ for any word $w\in\cH^1$ and $w_1e_{k_1}*w_2e_{k_2}=(w_1*w_2e_{k_2})e_{k_1}+(w_1e_{k_1}*w_2)e_{k_2}+(w_1*w_2)e_{k_1+k_2}$ for any words $w_1$, $w_2\in\cH^1$ and any positive integers $k_1$, $k_2$ with $\QQ$-bilinearity.
Let $\mathsf{LinKaw}$ be the space of the linear part of Kawashima's relations:
\[
\mathsf{LinKaw}\coloneqq\spa_{\QQ}\{\phi(w_1\ast w_2)x \mid w_1, w_2\in y\cH\}.
\]
This space is indeed a family of relations satisfied by multiple zeta values.
\begin{theorem}[{Kawashima~\cite[Corollary~5.5]{Kawashima2009}}]
$\mathsf{LinKaw}\subset\mathrm{Ker}(Z)$.
\end{theorem}
The space $\mathsf{LinKaw}$ does not provide all relations among multiple zeta values, but it is relatively large.
In fact, Kawashima \cite{Kawashima2009} proved that $\mathsf{LinKaw}$ contains Ohno's relations $\mathsf{Ohno}$ (\cite{Ohno1999}), which is a natural generalization of the duality relations $\mathsf{Duality}$, Tanaka--Wakabayashi \cite{TanakaWakabayashi2010} proved that $\mathsf{LinKaw}$ contains the cyclic sum formula $\mathsf{CycSum}$ (\cite{HoffmanOhno2003}), and Tanaka \cite{Tanaka2009} proved that $\mathsf{LinKaw}$ contains the quasi-derivation relations $\mathsf{QuaDer}$.
Relatedly, Ihara--Kaneko--Zagier \cite{IharaKanekoZagier2006} proved that the derivation relations $\mathsf{Der}$ are contained in Ohno's relations.
See \cref{fig:relations}.
Hirose--Murahara--Onozuka~\cite{HiroseMuraharaOnozuka2021} proved that a certain special family of relations coincides nontrivially with $\mathsf{LinKaw}$, thereby providing another proof that $\mathsf{Duality}\subset\mathsf{LinKaw}$ and $\mathsf{Der}\subset\mathsf{LinKaw}$.
They \cite{HiroseMuraharaOnozuka2023} further proved that the family of relations satisfied by parameterized multiple zeta values is exhausted by $\mathsf{LinKaw}$.
Bachmann--Tanaka~\cite{BachmannTanaka2020} proved that the independently motivated family known as the \emph{rooted tree map relations} also coincides nontrivially with $\mathsf{LinKaw}$.

We consider the following family of relations which is a slight extension of $\mathsf{LinKaw}$:
\[
\mathsf{LinKaw}^*\coloneqq\spa_{\QQ}\{(\phi(w_1\ast w_2)x)*w_3 \mid w_1, w_2\in y\cH, w_3\in\cH^{\geq 2}\}.
\]
By the harmonic product formula for multiple zeta values, this is also a subspace of $\mathrm{Ker}(Z)$.
The restriction on the range of $w_3$ is due to the fact that $\zeta^{\dia}$-values, defined in the next section, satisfy the restricted harmonic product formula (\cref{prop:dia-harmonic}).

Based on numerical experiments up to weight 17, we conjecture the following equality between families of relations.
\begin{conjecture}
$\mathsf{LinKaw}^*=\mathsf{Drop1}$.
\end{conjecture}
Regarding this equality, we prove one of the inclusions, which is the second main result.
\begin{theorem}\label{thm:main2}
$\mathsf{LinKaw}^*\subset\mathsf{Drop1}$.
\end{theorem}
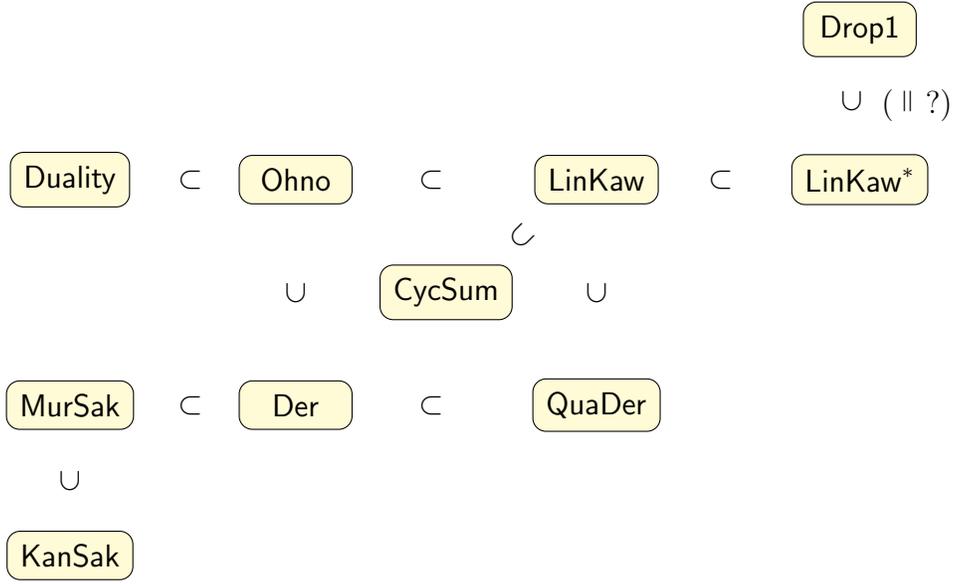
\begin{figure}[htbp]
\centering
\begin{tikzpicture}[font=\large]
  \tikzset{mynode/.style={
    draw, rectangle, rounded corners=5pt,
    inner sep=5pt, minimum width=1.5cm,
    fill=yellow!20
  }}
\node[mynode] (der) at (0,0) {$\mathsf{Der}$};
\node[mynode] (gender) at (4,0) {$\mathsf{QuaDer}$};
\node[mynode] (ohno) at (0,3) {$\mathsf{Ohno}$};
\node[mynode] (linkaw) at (4,3) {$\mathsf{LinKaw}$};
\node[mynode] (cycsum) at (2,1.5) {$\mathsf{CycSum}$};
\node at ($(cycsum)!0.5!(linkaw)$) {\rotatebox{45}{$\subset$}};
\node at ($(der)!0.5!(ohno)$) {\rotatebox{90}{$\subset$}};
\node [xshift=-0.2cm] at ($(der)!0.5!(gender)$) {$\subset$};
\node [xshift=-0.2cm] at ($(ohno)!0.5!(linkaw)$) {$\subset$};
\node at ($(gender)!0.5!(linkaw)$) {\rotatebox{90}{$\subset$}};
\node[mynode] (mursak) at (-3,0) {$\mathsf{MurSak}$};
\node [xshift= 0.1cm] at ($(mursak)!0.5!(der)$) {$\subset$};
\node[mynode] (kansak) at (-3,-2) {$\mathsf{KanSak}$};
\node at ($(mursak)!0.5!(kansak)$) {\rotatebox{90}{$\subset$}};
\node[mynode] (linkawstar) at (7.5,3) {$\mathsf{LinKaw}^*$};
\node [xshift=-0.1cm] at ($(linkaw)!0.5!(linkawstar)$) {$\subset$};
\node[mynode] (drop1) at (7.5,5) {$\mathsf{Drop1}$};
\node[xshift= 0.5cm] at ($(linkawstar)!0.5!(drop1)$) {\rotatebox{90}{$\subset$} \ ( \rotatebox{90}{$=$} ?)};
\node[mynode] (duality) at (-3,3) {$\mathsf{Duality}$};
\node [xshift= 0.1cm] at ($(duality)!0.5!(ohno)$) {$\subset$};
\end{tikzpicture}
  \caption{Relationship between $\mathsf{Drop1}$ and related families of relations}
  \label{fig:relations}
\end{figure}
The family of relations $\mathsf{LinKaw}$ by itself is much smaller than $\mathsf{Drop1}$.
In fact, given that
\begin{equation}\label{eq:dim_Drop1}
\dim_{\QQ}(\cH^0_k/(\mathsf{Drop1}\cap\cH^0_k))=F_{k-1} 
\end{equation}
(explained below), Kawashima (\cite[Corollary~6.5]{Kawashima2009}) showed that 
\[
\dim_{\QQ}(\cH^0_k/(\mathsf{LinKaw}\cap\cH^0_k))=\frac{1}{k-1}\sum_{d\mid k-1}\mu\left(\frac{k-1}{d}\right)2^d\sim\frac{2^{k-1}}{k}
\]
as $k\to\infty$ (\cite[A001037]{OEIS}).
Here, for each $k\geq 2$, $\cH^0_k$ is the subspace of $\cH^0$ generated by the homogeneous elements of total degree $k$, and $\mu$ denotes the M\"obius function.
Therefore, $\mathsf{LinKaw}$ alone does not imply that $\dim_{\QQ}\cZ_k=O(a^k)$ for any $a<2$.

Let $\cH^{\geq 2}_k\coloneqq\cH^{\geq 2}\cap\cH^0_k$.
Since $\cD(w) = w$ holds for any $w \in \cH^{\geq 2}_k$ (this follows from \eqref{eq:dia-D} and \cref{lem:lifting}), we see that $\mathsf{Drop1} \cap \cH^{\geq 2}_k = \{0\}$.
Therefore, the composition $\cH^{\geq 2}_k\hookrightarrow\cH^0_k\twoheadrightarrow\cH^0_k/(\mathsf{Drop1}\cap\cH^0_k)$ is an isomorphism, and thus \eqref{eq:dim_Drop1} holds.
\subsection{Organization}
The present paper is organized as follows.
In \cref{sec:outline}, we introduce the key ingredient, the $\zeta^{\dia}$-values, and explain how the two main theorems are derived from results for $\zeta^{\dia}$-values (namely, \cref{thm:dia-main2} and \cref{thm:dia-Kawashima}).
In \cref{sec:dia}, we prove the fundamental properties of $\zeta^{\dia}$-values, namely the restricted harmonic product formula, an expression of the star value, and a discrete iterated integral expression.
In \cref{sec:difference}, we prove \cref{thm:dia-main2}, which provides our main algorithm.
In \cref{sec:Kawashima}, we prove Kawashima's relations for $\zeta^{\dia}$-values (namely, \cref{thm:dia-Kawashima}).
\subsection*{Acknowledgments}
The third author would like to thank Professor Yoshinori Mishiba for suggesting that one pursue a characteristic zero analogue of Ngo Dac's method in \cite{NgoDac2021} based on the MSW formula~\cite{MaesakaSekiWatanabe-pre}.
The present paper offers a partial answer to this suggestion.
The authors would like to thank Professor Henrik Bachmann for kindly answering my question regarding the content of \cref{rem:qMZV}.
The authors would also like to thank Professor Hidekazu Furusho for kindly informing us about Ecalle’s work.
The authors would like to thank Hanamichi Kawamura, Wataru Machida, and Professor Wadim Zudilin for their comments on the paper.

\section{\texorpdfstring{An outline of the proof and $\zeta^{\dia}$}{An outline of the proof and dia-values}}\label{sec:outline}
For a positive integer $N$ and an index $\bk=(k_1,\dots, k_r)$, the \emph{multiple harmonic sum} $\zeta^{}_N(\bk)$ is defined as
\[
\zeta^{}_N(\bk)\coloneqq \sum_{0<n_1<\cdots<n_r<N}\frac{1}{n_1^{k_1}\cdots n_r^{k_r}}.
\]
Let $\zeta^{\rH}(\bk)$ denote $(\zeta^{}_N(\bk))_{N\in\NN}\in\QQ^{\NN}$.
Here, $\NN$ is the set of positive integers.
For each integer $k\geq 2$, let $\cZ_k^{\rH}$ denote the subspace $\spa_{\QQ}\{\zeta^{\rH}(\bk) \mid \bk\in\II^{\mathrm{adm}}_k\}$ of $\QQ^{\NN}$.
In contrast to the case of multiple zeta values, all $\zeta^{\rH}(\bk)$ are linearly independent by the following lemma.
\begin{lemma}\label{lem:lifting}
The $\QQ$-linear mapping $Z^{\rH}\colon\cH^1\to\QQ^{\NN}$, which is defined by $Z^{\rH}(1)=1$ and $Z^{\rH}(e_{k_1}\cdots e_{k_r})=\zeta^{\rH}(\bk)$ for each index $\bk=(k_1,\dots, k_r)$, is injective. 
\end{lemma}
\begin{proof}
This is stated in \cite[Theorem~3.1]{Yamamoto2013}, which is proved based on \cite[Corollary~5.6]{Brown2009}.
\end{proof}
In particular, $\{\zeta^{\rH}(\bk) \mid \bk \in \II^{\mathrm{adm}}_k\}$ is a basis of $\cZ_k^{\rH}$ and we have $\dim_{\QQ}\cZ_k^{\rH}=2^{k-2}$.
Since a surjection $\cZ_k^{\rH}\twoheadrightarrow \cZ_k$ given by $\zeta^{\rH}(\bk)\mapsto\zeta(\bk)$ is obtained as the restriction of the mapping that sends a Cauchy sequence to its limit, the dimension estimate $\dim_{\QQ}\cZ_k\leq \dim_{\QQ}\cZ_k^{\rH}=2^{k-2}$ follows.
Of course, this is trivial.
\begin{remark}
In the case of positive characteristic multiple zeta values, there exist nontrivial linear relations among multiple harmonic sums (cf.~\cite[Example~3.18]{Todd2018}).
From this, it follows that attempting to solve \cref{prob:fundamental} using exactly the same approach as \cite{NgoDac2021} does not work.
\end{remark}
Instead of $\zeta^{\rH}(\bk)$, which satisfy no linear relations at all, we introduce $\zeta^{\dia}(\bk)\in\QQ^{\NN}$, which partially satisfy the same relations as multiple zeta values.
We establish a setting where the subspace $\cZ^{\dia}_k$ of $\QQ^{\NN}$ spanned by these values of weight $k$ has dimension $F_{k-1}$, and where there exists a natural surjection from $\cZ^{\dia}_k$ to $\cZ_k$.
\begin{definition}\label{def:dia}
For a positive integer $N$ and an admissible index $\bk=(k_1,\dots, k_r)$, we define $\zeta^{\dia}_N(\bk)$ as
\[
\zeta^{\dia}_{N}(\bk)\coloneqq\sum_{A\subset [r]^1_{\bk}}\sum_{(n_1,\dots,n_r)\in S_{r,N}(A)}\Biggl(\prod_{i\in A}\frac{1}{N-n_i}\Biggr)\Biggl(\prod_{i\in[r]\setminus A}\frac{1}{n_i^{k_i}}\Biggr),
\]
where $[r]^1_{\bk}\coloneqq \{i\in[r] \mid k_i=1\}$ and
\[
S_{r,N}(A) \coloneqq \left\{(n_1,\dots,n_r) \in [N-1]^r \ \middle| \ \begin{array}{cc} 
n_i\leq n_{i+1}  & \text{ if } i\in A, \\ 
n_i < n_{i+1} & \text{ if } i\in [r-1]\setminus A
\end{array} \right\}
.
\]
Let $\zeta^{\dia}(\bk)$ denote $(\zeta^{\dia}_N(\bk))_{N\in\NN}\in\QQ^{\NN}$.
The $\QQ$-linear mapping $Z^{\dia}\colon\cH^0\to\QQ^{\NN}$ is defined by $Z^{\dia}(1)=1$ and $Z^{\dia}(e_{k_1}\cdots e_{k_r})=\zeta^{\dia}(\bk)$ for each admissible index $\bk=(k_1,\dots, k_r)$.
For an integer $k\geq 2$, let $\cZ^{\dia}_k=\spa_{\QQ}\{ \zeta^{\dia}(\bk) \mid \bk\in\II_k^{\mathrm{adm}}\}$ denote the $\QQ$-vector space spanned by all $\zeta^{\dia}$-values of weight $k$.
\end{definition}
If $\bk\in\II^{\geq 2}_k$, then we have $\zeta^{\dia}_N(\bk)=\zeta^{}_N(\bk)$ because $[r]^1_{\bk}=\varnothing$ with $r=\dep(\bk)$.
For any admissible index $\bk$, in 
\[
\zeta^{\dia}_{N}(\bk)=\zeta^{}_N(\bk)+\sum_{\varnothing\neq A\subset [r]^1_{\bk}}\sum_{(n_1,\dots,n_r)\in S_{r,N}(A)}\Biggl(\prod_{i\in A}\frac{1}{N-n_i}\Biggr)\Biggl(\prod_{i\in[r]\setminus A}\frac{1}{n_i^{k_i}}\Biggr),
\]
it follows that 
\[
\lim_{N\to\infty}\sum_{\varnothing\neq A\subset [r]^1_{\bk}}\sum_{(n_1,\dots,n_r)\in S_{r,N}(A)}\Biggl(\prod_{i\in A}\frac{1}{N-n_i}\Biggr)\Biggl(\prod_{i\in[r]\setminus A}\frac{1}{n_i^{k_i}}\Biggr)=0
\]
by the same mechanism as in the proof of \cite[Lemma~2.2 (ii)]{Seki2025}, and thus 
\[
\lim_{N\to\infty}\zeta^{\dia}_N(\bk)=\zeta(\bk)
\]
holds.
Therefore, a natural surjection $\cZ_k^{\dia}\twoheadrightarrow\cZ_k$ given by $\zeta^{\dia}(\bk)\mapsto \zeta(\bk)$ is obtained as the restriction of the mapping that sends a Cauchy sequence to its limit.
Note that the argument involving limits is confined to this; all other arguments in this paper involve only elementary manipulations of finite sums.
In this respect, the methodology of the paper is thoroughly elementary (except for \cref{lem:lifting}, which is treated as a black box).

Unlike the case of multiple harmonic sums, $\zeta^{\dia}$-values can satisfy the same relations as multiple zeta values.
For example, Euler's relation
\[
\zeta^{\dia}(1,2)=\zeta^{\dia}(3)
\]
holds.
In fact,
\[
\sum_{0<n_1<n_2<N}\frac{1}{n_1n_2^2}+\sum_{0<n_1\leq n_2<N}\frac{1}{(N-n_1)n_2^2}=\sum_{n=1}^{N-1}\frac{1}{n^3},
\]
which is nothing but \cite[Equation~(3.1)]{BorweinBradley2006} (see also \cite[Section~2]{MaesakaSekiWatanabe-pre}).
Any relation satisfied by $\zeta^{\dia}$-values must also be satisfied by multiple zeta values when taking the limit $N\to\infty$, but the converse does not generally hold.

Since elements of $\{\zeta^{\dia}(\bk) \mid \bk\in\II^{\geq 2}_k\}=\{\zeta^{\rH}(\bk) \mid \bk\in\II^{\geq 2}_k\}$ are linearly independent by \cref{lem:lifting}, unlike the case of multiple zeta values, we can say that $\dim_{\QQ}\cZ_k^{\dia}\geq F_{k-1}$.
We determine this dimension as follows.
\begin{theorem}
Let $k\geq 2$ be an integer.
The set $\{\zeta^{\dia}(\bk) \mid \bk\in\II^{\geq 2}_k\}$ is a basis of $\cZ_k^{\dia}$.
In particular, we have
\[
\dim_{\QQ}\cZ_k^{\dia}=F_{k-1}.
\]
\end{theorem}
This follows from the following stronger theorem.
\begin{theorem}\label{thm:dia-main2}
Let $k\geq 2$ be an integer.
For every $\bk \in \II^{\mathrm{adm}}_k$ and $\bl\in\II^{\geq 2}_k$, there exists a unique integer $c_{\bk;\bl}$ such that for any $\bk \in \II^{\mathrm{adm}}_k$,
\[
\zeta^{\dia}(\bk)=\sum_{\bl\in\II^{\geq 2}_k}c_{\bk;\bl}\zeta^{\dia}(\bl)
\]
holds.
The integers $c_{\bk;\bl}$ are given by the formula
\begin{equation}\label{eq:dia-D}
\zeta^{\dia}(\{1\}^{c_1-1},c_2+1,\dots, \{1\}^{c_{2s-1}-1},c_{2s}+1)=Z^{\dia}(\fD(c_1,\dots, c_{2s}))
\end{equation}
for positive integers $s$ and $c_1,\dots, c_{2s}$.
\end{theorem}
The uniqueness of $c_{\bk;\bl}$ follows from \cref{lem:lifting}, and the remaining part is proved in \cref{sec:difference}.
Taking the limit as $N\to\infty$ in \cref{thm:dia-main2} yields \cref{thm:main1} and \cref{thm:the-algorithm}.
The formula \eqref{eq:dia-D} indicates that the equality holds between finite sums for each $N$, and its proof is achieved by naturally performing difference calculations with respect to $N$. 
The ability to perform such difference calculations is one of the advantages of approaching the problem through finite sums.
In the difference calculations, the $\dia$-version (\cref{cor:dia-discrete_integral}) of the discrete iterated integral expression obtained in \cite{MaesakaSekiWatanabe-pre} plays a key role.
\begin{table}[h]
\centering
\small
\begin{tabular}{|>{\centering\arraybackslash}m{2.8cm}|>{\centering\arraybackslash}m{3.7cm}|>{\centering\arraybackslash}m{3.7cm}|>{\centering\arraybackslash}m{3.7cm}|}
\hline
\textbf{Space} & \( \mathcal{Z}_k^{\rH} \) & \( \mathcal{Z}_k^{\mathrm{\dia}} \) & \( \mathcal{Z}_k \) \\ \hline
\textbf{Dimension} & \( 2^{k-2} \) & \( F_{k-1} \) & \( d_k \) \\ \hline
\textbf{Order} & \( 2^k \) & \( 1.6180^k \) & \( 1.3247^k \) \\ \hline
\textbf{Basis} & \( \{\zeta^{\rH}(\bk) \mid \bk \in \mathbb{I}^{\mathrm{adm}}_k\} \) & \( \{\zeta^{\dia}(\bk) \mid \bk \in \mathbb{I}^{\geq 2}_k\} \) & \( \{\zeta(\bk) \mid \bk \in \mathbb{I}^{\rH}_k\} \) \\ \hline
\textbf{Status} & proved (\cref{lem:lifting}) & proved (this paper) & conjecture \\ \hline
\addlinespace
\end{tabular}
\caption{Summary of spaces and their properties}
\end{table}

Furthermore, we can determine the kernel of $Z^{\dia}$ from \cref{thm:dia-main2}.
\begin{theorem}\label{thm:dia-Drop1}
The family of relations obtained by \eqref{eq:dia-D} exhaust all the $\QQ$-linear relations among $\zeta^{\dia}$-values, that is, 
\[
\mathrm{Ker}(Z^{\dia})=\mathsf{Drop1}.
\]
\end{theorem}
\begin{proof}
Since $\mathsf{Drop1}\subset\mathrm{Ker}(Z^{\dia})$ holds by \eqref{eq:dia-D}, the mapping $Z^{\dia}$ induces $\overline{Z^{\dia}}\colon \cH^0/\mathsf{Drop1} \to \QQ^{\NN}$.
Since $\cH^0/\mathsf{Drop1}=\{w\bmod{\mathsf{Drop1}}\mid w\in \cH^{\geq 2}\}$ and $Z^{\dia}$ is injective on $\cH^{\geq 2}$ by \cref{lem:lifting}, we see that $\overline{Z^{\dia}}$ is an injection.
Therefore, we have $\mathrm{Ker}(Z^{\dia})=\mathsf{Drop1}$.
\end{proof}
Although $\mathsf{Drop1}$ is a new family of relations, we prove that the following relations, already known for multiple zeta values, are also satisfied by $\zeta^{\dia}$-values.
We define a $\QQ$-bilinear product $\circledast$ on $y\cH$ by $w_1e_{k_1}\circledast w_2e_{k_2}=(w_1*w_2)e_{k_1+k_2}$ for any words $w_1$, $w_2\in\cH^1$ and any positive integers $k_1$, $k_2$.
\begin{theorem}\label{thm:dia-Kawashima}
For a positive integer $m$ and $w_1$, $w_2\in y\cH$, we have
\[
\sum_{\substack{a, b\in \NN \text{ s.t.} \\ a+b=m}}Z^{\dia}(\phi(w_1)\circledast y^a)Z^{\dia}(\phi(w_2)\circledast y^b)=Z^{\dia}(\phi(w_1*w_2)\circledast y^m).
\]
\end{theorem}
Replacing $Z^{\dia}$ with $Z$ yields Kawashima's relations for multiple zeta values, and it is conjectured that the family of linear relations obtained by expanding the product on the left-hand side using the shuffle product formula provides all relations among multiple zeta values (\cite{Kawashima2009}).
In the case of $Z^{\dia}$, however, the shuffle product formula does not hold.
This theorem will be proved in \cref{sec:Kawashima} by using the connected sum method due to Seki--Yamamoto~\cite{SekiYamamoto2019}.
\begin{corollary}\label{cor:dia-linear-Kawashima}
$\mathsf{LinKaw}^*\subset\mathrm{Ker}(Z^{\dia})$.
\end{corollary}
\begin{proof}
The case $m=1$ of \cref{thm:dia-Kawashima} shows $\mathsf{LinKaw}\subset\mathrm{Ker}(Z^{\dia})$, and it suffices to combine this with \cref{prop:dia-harmonic}, which will be proved later.
\end{proof}
\cref{thm:main2} follows from \cref{thm:dia-Drop1} and \cref{cor:dia-linear-Kawashima}.
\cref{thm:main2} is an inclusion relation between two families of abstract relations formulated on the Hoffman algebra, and it is not a claim concerning concrete realizations such as multiple zeta values.
Several studies have investigated inclusion relationships among various families of abstract relations.
In addition to the studies in \cref{fig:relations}, there are, for instance, \cite{Furusho2011}, \cite{Furusho2022}, \cite[Theorem~28 and Theorem~30]{HiroseSato2019}, and \cite[Theorem~4.6]{KanekoYamamoto2018}.
Furthermore, some unresolved problems remain.
For example, it is still unknown whether the duality relations are contained in the extended double shuffle relations (cf.~\cite{Kimura2024}).
Proving that concrete objects like multiple zeta values satisfy two families of abstract relations does not, by itself, imply an inclusion relation between these families.
However, if one of the two families coincides with the kernel of a mapping that realizes specific concrete objects (in this case, $Z^{\dia}$), then by proving that these specific objects satisfy the other family, an inclusion relation on the Hoffman algebra is obtained, as illustrated by the proof of \cref{thm:main2}.
In our work, the equality between the kernel and $\mathsf{Drop1}$ is proved thanks to \cref{lem:lifting}, and this ability to establish such an argument represents another advantage of the finite sum approach. 

\section{\texorpdfstring{Basic properties of $\zeta^{\dia}$}{Basic properties of dia-values}}\label{sec:dia}
For a non-negative integer $n$, $[n]_0\coloneqq [n]\cup\{0\}$.
For an index $\bk=(k_1,\dots, k_r)$ and a subset $A\subset [r-1]$, denote by $\bk_{(+A)}$ the index obtained from $(k_1 \, \Box\cdots \Box \, k_r)$ by reading each $k_i \, \Box \, k_{i+1}$ as $k_i+k_{i+1}$ when $i\in A$ and as $k_i, k_{i+1}$ (using a comma between $k_i$ and $k_{i+1}$) when $i\notin A$.
See \cref{subsec:1st-main} for the notation $\bk_{(-A)}$.
\subsection{Harmonic product formula}
The following proposition is clear from the harmonic product formula for multiple harmonic sums when both $w_1$ and $w_2$ are elements of $\cH^{\geq 2}$.
However, it still holds even when one of them is an element of $\cH^0$.
\begin{proposition}\label{prop:dia-harmonic}
For any $w_1\in\cH^0$ and $w_2\in\cH^{\geq 2}$, we have
\[
Z^{\dia}(w_1\ast w_2)=Z^{\dia}(w_1)Z^{\dia}(w_2).
\]
\end{proposition}
\begin{proof}
Let $\bk=(k_1,\dots,k_r)$ be an admissible index and let $\bl=(k_{r+1},\cdots, k_{r+s})$ be an index such that $k_i\geq 2$ for $i\in[r+s]\setminus[r]$.
For each $d\in[\min\{r,s\}]_0$, let $\Sigma_{r,s,d}$ denote the set of all surjections $\sigma\colon[r+s]\twoheadrightarrow [r+s-d]$ satisfying $\sigma(1)<\sigma(2)<\cdots<\sigma(r)$ and $\sigma(r+1)<\sigma(r+2)<\cdots<\sigma(r+s)$.
For $\sigma\in\Sigma_{r,s,d}$, we define an index $(k_1^{\sigma},\dots, k_{r+s-d}^{\sigma})$ by
\[
k_j^{\sigma}\coloneqq\begin{cases}k_a & \text{if }\sigma^{-1}(\{j\})=\{a\}, \\ k_a+k_b & \text{if }\sigma^{-1}(\{j\})=\{a,b\} \text{ with } a\neq b. \end{cases}
\]
For each subset $A$ of $[r]^1_{\bk}$, we have
\begin{align*}
&\Biggl(\sum_{(n_1,\dots,n_r)\in S_{r,N}(A)}\Biggl(\prod_{i\in A}\frac{1}{N-n_i}\Biggr)\Biggl(\prod_{i\in[r]\setminus A}\frac{1}{n_i^{k_i}}\Biggr)\Biggr)\zeta^{}_N(\bl)\\
&=\sum_{d=0}^{\min\{r,s\}}\sum_{\substack{\sigma\in\Sigma_{r,s,d} \\ \#\sigma^{-1}(\{j\})=1 \ (j\in\sigma(A))}}\sum_{(m_1,\dots, m_{r+s-d})\in S_{r+s-d,N}(\sigma(A))}\Biggl(\prod_{j\in \sigma(A)}\frac{1}{N-m_j}\Biggr)\\
&\qquad\times\Biggl(\prod_{j\in [r+s-d]\setminus \sigma(A)}\frac{1}{m_j^{k^{\sigma}_j}}\Biggr).
\end{align*}
For each fixed $\sigma$, we have
\begin{align*}
&\sum_{\substack{A\subset[r]^1_{\bk} \\ \#\sigma^{-1}(\{j\})=1 \ (j\in\sigma(A))}}\sum_{(m_1,\dots, m_{r+s-d})\in S_{r+s-d,N}(\sigma(A))}\Biggl(\prod_{j\in \sigma(A)}\frac{1}{N-m_j}\Biggr)\Biggl(\prod_{j\in [r+s-d]\setminus \sigma(A)}\frac{1}{m_j^{k^{\sigma}_j}}\Biggr)\\
&=\zeta^{\dia}_N(k_1^{\sigma},\dots,k_{r+s-d}^{\sigma}).
\end{align*}
Therefore, we see that
\[
\zeta^{\dia}_N(\bk)\zeta_N(\bl)=\sum_{d=0}^{\min\{r,s\}}\sum_{\sigma\in \Sigma_{r,s,d}}\zeta^{\dia}_N(k_1^{\sigma},\dots,k^{\sigma}_{r+s-d})
\]
holds, thereby yielding the harmonic product formula.
\end{proof}
\begin{example}
For integers $k$, $l\geq 2$,
\begin{align*}
\zeta^{\dia}_N(1,k)\zeta^{\dia}_N(l)&=\left(\sum_{0<n_1<n_2<N}\frac{1}{n_1n_2^k}+\sum_{0<n_1\leq n_2<N}\frac{1}{(N-n_1)n_2^k}\right)\left(\sum_{m=1}^{N-1}\frac{1}{m^l}\right)\\
&=\sum_{0<m<n_1<n_2<N}\frac{1}{m^ln_1n_2^k}+\sum_{0<n_1<m<n_2<N}\frac{1}{n_1m^ln_2^k}+\sum_{0<n_1<n_2<m<N}\frac{1}{n_1n_2^km^l}\\
&\quad +\sum_{0<n_1<n_2<N}\frac{1}{n_1^{l+1}n_2^k}+\sum_{0<n_1<n_2<N}\frac{1}{n_1n_2^{k+l}}\\
&\quad +\sum_{0<m<n_1\leq n_2<N}\frac{1}{m^l(N-n_1)n_2^k}+\sum_{0<n_1\leq m<n_2<N}\frac{1}{(N-n_1)m^ln_2^k}\\
&\quad + \sum_{0<n_1\leq n_2<m<N}\frac{1}{(N-n_1)n_2^km^l}+\sum_{0<n_1\leq n_2<N}\frac{1}{(N-n_1)n_2^{k+l}}\\
&=\left(\sum_{0<m<n_1<n_2<N}\frac{1}{m^ln_1n_2^k}+\sum_{0<m<n_1\leq n_2<N}\frac{1}{m^l(N-n_1)n_2^k}\right)\\
&\quad + \left(\sum_{0<n_1<m<n_2<N}\frac{1}{n_1m^ln_2^k}+\sum_{0<n_1\leq m<n_2<N}\frac{1}{(N-n_1)m^ln_2^k}\right)\\
&\quad +\left(\sum_{0<n_1<n_2<m<N}\frac{1}{n_1n_2^km^l}+\sum_{0<n_1\leq n_2<m<N}\frac{1}{(N-n_1)n_2^km^l}\right)\\
&\quad +\sum_{0<n_1<n_2<N}\frac{1}{n_1^{l+1}n_2^k}\\
&\quad +\left(\sum_{0<n_1<n_2<N}\frac{1}{n_1n_2^{k+l}}+\sum_{0<n_1\leq n_2<N}\frac{1}{(N-n_1)n_2^{k+l}}\right)\\
&=\zeta^{\dia}_N(l,1,k)+\zeta^{\dia}_N(1,l,k)+\zeta^{\dia}_N(1,k,l)+\zeta^{\dia}_N(l+1,k)+\zeta^{\dia}_N(1,k+l).
\end{align*}
The following example shows that the $\zeta^{\dia}$-values do not satisfy the harmonic product formula in general.
\begin{align*}
\zeta^{\dia}_N(1,2)^2&=2\zeta^{\dia}_N(1,2,1,2)+4\zeta^{\dia}_N(1,1,2,2)+2\zeta^{\dia}_N(1,1,4)+2\zeta^{\dia}_N(1,3,2)+2\zeta^{\dia}_N(2,2,2)+\zeta^{\dia}_N(2,4)\\
&\quad-2\sum_{0<n_1\leq n_2<n_3<N}\frac{1}{(N-n_1)^2n_2^2n_3^2}-\sum_{0<n_1\leq n_2<N}\frac{1}{(N-n_1)^2n_2^4}.
\end{align*}
Nevertheless, by applying \cref{thm:dia-main2}, the product of two $\zeta^{\dia}$-values can always be expressed as a linear combination of $\zeta^{\dia}$-values.
\end{example}
\subsection{Star value}
Analogously to other objects like multiple zeta values, we define the star version of the $\zeta^{\dia}$-value by
\[
\zeta^{\dia\star}_N(\bk)\coloneqq\sum_{\bl\preceq\bk}\zeta^{\dia}_N(\bl).
\]
This value can also be expressed as follows.
\begin{proposition}\label{prop:star-expression}
For a positive integer $N$ and an admissible index $\bk=(k_1,\dots, k_r)$, we have
\[
\zeta_N^{\dia\star}(\bk)=\sum_{A\subset [r]^1_{\bk}}\sum_{(n_1,\dots,n_r)\in S^{\star}_{r,N}(A)}\Biggl(\prod_{i\in A}\frac{1}{N-n_i}\Biggr)\Biggl(\prod_{i\in[r]\setminus A}\frac{1}{n_i^{k_i}}\Biggr),
\]
where
\[
S^{\star}_{r,N}(A) \coloneqq \left\{(n_1,\dots,n_r) \in [N-1]^r \ \middle| \ \begin{array}{cc} 
n_i\leq n_{i+1}  & \text{ if } i\in A\cap[r-1] \text{ or } i+1\in[r]\setminus A, \\ 
n_i < n_{i+1} & \text{ if } i\in[r-1]\setminus A \text{ and } i+1\in A
\end{array} \right\}
.
\]
\end{proposition}
\begin{proof}
For $B\subset [r-1]$, we denote by $B^+$ the set $\{i+1\in[r] \mid i\in B\}$.
By definition, we compute
\begin{align*}
\zeta^{\dia\star}_N(\bk)&=\sum_{B\subset[r-1]}\zeta^{\dia}_N(\bk_{(+B)})\\
&=\sum_{B\subset [r-1]}\sum_{A\subset[r]^1_{\bk}\setminus(B\cup B^+)}\sum_{\substack{(n_1,\dots,n_r)\in[N-1]^r \\ n_i\leq n_{i+1} \ (i\in A) \\ n_i=n_{i+1} \ (i\in B) \\ n_i<n_{i+1} \ (i\in[r-1]\setminus (A\cup B))}}\Biggl(\prod_{i\in A}\frac{1}{N-n_i}\Biggr)\Biggl(\prod_{i\in[r]\setminus A}\frac{1}{n_i^{k_i}}\Biggr)\\
&=\sum_{A\subset[r]^1_{\bk}}\sum_{\substack{B\subset [r-1] \\ A\cap (B\cup B^+)=\varnothing}}\sum_{\substack{(n_1,\dots,n_r)\in[N-1]^r \\ n_i\leq n_{i+1} \ (i\in A) \\ n_i=n_{i+1} \ (i\in B) \\ n_i<n_{i+1} \ (i\in[r-1]\setminus (A\cup B))}}\Biggl(\prod_{i\in A}\frac{1}{N-n_i}\Biggr)\Biggl(\prod_{i\in[r]\setminus A}\frac{1}{n_i^{k_i}}\Biggr).
\end{align*}
Since the inner two summations are equal to $\sum_{(n_1,\dots,n_r)\in S^{\star}_{r,N}(A)}$, the proof is complete.
\end{proof}
\subsection{Discrete iterated integral expression}
Based on the idea of the discrete iterated integral expression of multiple harmonic sums given in \cite{MaesakaSekiWatanabe-pre}, we provide a discrete iterated integral expression of $\zeta^{\dia}_N$-values.
\begin{definition}
For a positive integer $N$ and an index $\bk=(k_1,\dots,k_r)$, we define $\zeta^{\flat}_N(\bk)$ as
\[
\zeta^{\flat}_N(\bk)\coloneqq \sum_{\substack{0<n_{i,1}\leq \cdots \leq n_{i,k_i}<N \ (i\in [r]) \\ n_{i,k_i}<n_{i+1,1} \ (i\in[r-1])}}\prod_{i\in[r]}\frac{1}{(N-n_{i,1})n_{i,2}\cdots n_{i,k_i}}.
\]
\end{definition}
If $\bk$ is admissible, then $\displaystyle\lim_{N\to\infty}\zeta^{\flat}_N(\bk)$ gives the iterated integral expression of the multiple zeta value $\zeta(\bk)$.
That is, $\displaystyle \lim_{N\to\infty}\zeta^{}_N(\bk)=\lim_{N\to\infty}\zeta^{\flat}_N(\bk)$ holds; in fact, the equality already holds before taking the limit.
\begin{theorem}[Maesaka--Seki--Watanabe~\cite{MaesakaSekiWatanabe-pre}]\label{thm:MSW}
For every positive integer $N$ and every index $\bk$, we have
\[
\zeta^{}_N(\bk)=\zeta^{\flat}_N(\bk).
\]
\end{theorem}
This formula can be generalized as follows.
\begin{definition}
Let $\bk=(k_1,\dots, k_r)$ and $\bl=(l_1,\dots, l_r)$ be indices of the same depth.
For a positive integer $N$, we define $\zeta^{}_N\left(\genfrac{}{}{0pt}{}{\bl}{\bk}\right)$ and $\zeta^{\flat}_N\left(\genfrac{}{}{0pt}{}{\bl}{\bk}\right)$ as
\[
\zeta^{}_N\left(\genfrac{}{}{0pt}{}{\bl}{\bk}\right)\coloneqq\sum_{\substack{0<m_{j,1}\leq \cdots \leq m_{j,l_j}< N \ (j\in[r]) \\ m_{j,l_j}<m_{j+1,1} \ (j\in[r-1])}}\prod_{j\in[r]}\frac{1}{(N-m_{j,1})\cdots (N-m_{j,l_{j-1}})m_{j,l_j}^{k_j}}
\]
and
\[
\zeta^{\flat}_N\left(\genfrac{}{}{0pt}{}{\bl}{\bk}\right)\coloneqq \sum_{\substack{0<n_{i,1}\leq \cdots \leq n_{i,k_i}<N \ (i\in [r]) \\ n_{i,k_i}<n_{i+1,1} \ (i\in[r-1])}}\prod_{i\in[r]}\frac{1}{(N-n_{i,1})^{l_i}n_{i,2}\cdots n_{i,k_i}},
\]
respectively.
\end{definition}
\begin{theorem}\label{thm:flat-duality}
For every positive integer $N$ and every pair of indices $\bk$, $\bl$ of the same depth, we have
\[
\zeta^{}_N\left(\genfrac{}{}{0pt}{}{\bl}{\bk}\right)=\zeta^{\flat}_N\left(\genfrac{}{}{0pt}{}{\bl}{\bk}\right).
\]
\end{theorem}
\begin{remark}
The case $\bl=\{1\}^{\dep(\bk)}$ of \cref{thm:flat-duality} is \cref{thm:MSW}.
By applying the transformation of the form $n\mapsto N-n$, this theorem can also be expressed as a certain kind of duality:
\[
\zeta^{}_N\left(\genfrac{}{}{0pt}{}{l_1,\dots, l_r}{k_1,\dots, k_r}\right) = \zeta^{}_N\left(\genfrac{}{}{0pt}{}{k_r,\dots, k_1}{l_r,\dots, l_1}\right).
\]
As can be seen from the proof below, this theorem is equivalent to the discrete integral expression of the multiple polylogarithm proved by Hirose--Matsusaka--Seki.
\end{remark}
\begin{proof}
For indeterminates $x_1, \dots, x_r$ with $x_{r+1}=1$, an identity
\begin{align*}
&\sum_{0<m_1<\cdots<m_r<N}\frac{1}{m_1^{k_1}\cdots m_r^{k_r}}\prod_{j\in[r]}\frac{\binom{Nx_{j+1}-1}{m_j}}{\binom{Nx_j-1}{m_j}}\\
&=\sum_{\substack{0<n_{j,1}\leq \cdots \leq n_{j,k_j}<N \ (j\in[r]) \\ n_{j,k_j}<n_{j+1,1} \ (j\in[r-1])}}\prod_{j\in[r]}\frac{1}{(Nx_j-n_{j,1})n_{j,2}\cdots n_{j,k_j}}
\end{align*}
holds (\cite[Theorem~1.2]{HiroseMatusakaSeki-pre}).
Let $t_j\coloneqq N-Nx_j$ for each $j\in[r]$.
Since
\begin{align*}
\prod_{m_{j-1}<m\leq m_j}\frac{N-m}{Nx_j-m}&=\prod_{m_{j-1}<m\leq m_j}\frac{N-m}{N-m-t_j}\\
&=\sum_{l=1}^{\infty}\sum_{m_{j-1}<m_{j,1}\leq \cdots \leq m_{j,l-1}\leq m_{j}}\frac{t_j^{l-1}}{(N-m_{j,1})\cdots (N-m_{j,l-1})},
\end{align*}
we have
\begin{align*}
&\sum_{0<m_1<\cdots<m_r<N}\frac{1}{m_1^{k_1}\cdots m_r^{k_r}}\prod_{j\in[r]}\frac{\binom{Nx_{j+1}-1}{m_j}}{\binom{Nx_j-1}{m_j}}\\
&=\sum_{0<m_1<\cdots<m_r<N}\frac{1}{m_1^{k_1}\cdots m_r^{k_r}}\prod_{j\in[r]}\prod_{m_{j-1}<m\leq m_j}\frac{N-m}{Nx_j-m}\\
&=\sum_{l_1=1}^{\infty}\cdots \sum_{l_r=1}^{\infty}\zeta^{}_N\left(\genfrac{}{}{0pt}{}{l_1,\dots, l_r}{k_1,\dots, k_r}\right)t_1^{l_1-1}\cdots t_r^{l_r-1}.
\end{align*}
When $j=0$, read $m_{j-1}$ as 0.

On the other hand, by the expansion
\[
\frac{1}{Nx_j-n_{j,1}}=\sum_{l=1}^{\infty}\frac{t_j^{l-1}}{(N-n_{j,1})^{l}},
\]
we have
\begin{align*}
&\sum_{\substack{0<n_{j,1}\leq \cdots \leq n_{j,k_j}<N \ (j\in[r]) \\ n_{j,k_j}<n_{j+1,1} \ (j\in[r-1])}}\prod_{j\in[r]}\frac{1}{(Nx_j-n_{j,1})n_{j,2}\cdots n_{j,k_j}}\\
&=\sum_{l_1=1}^{\infty}\cdots \sum_{l_r=1}^{\infty}\zeta^{\flat}_N\left(\genfrac{}{}{0pt}{}{l_1,\dots, l_r}{k_1,\dots, k_r}\right)t_1^{l_1-1}\cdots t_r^{l_r-1}.
\end{align*}
Hence, the conclusion is obtained by comparing coefficients.
\end{proof}
Using \cref{thm:flat-duality}, we can also obtain a discrete iterated integral expression of $\zeta^{\dia}_N(\bk)$.
\begin{definition}
For a positive integer $N$ and an admissible index $\bk=(k_1,\dots, k_r)$, we define $\zeta^{\dia\flat}_N(\bk)$ as 
\[
\zeta^{\dia\flat}_N(\bk)\coloneqq \sum_{\substack{0<n_{i,1}\leq \cdots \leq n_{i,k_i}<N \ (i\in [r]) \\ n_{i,k_i}\leq n_{i+1,1} \ (i\in[r]^1_{\bk}) \\ n_{i,k_i}<n_{i+1,1} \ (i\in[r-1]\setminus [r]^1_{\bk})}}\prod_{i\in[r]}\frac{1}{(N-n_{i,1})n_{i,2}\cdots n_{i,k_i}}.
\]
\end{definition}
\begin{corollary}[Discrete iterated integral expression of $\zeta^{\dia}_N(\bk)$]\label{cor:dia-discrete_integral}
For every admissible index $\bk$, we have
\[
\zeta^{\dia}_N(\bk)=\zeta^{\dia\flat}_N(\bk).
\]
\end{corollary}
\begin{proof}
From the definitions and \cref{thm:flat-duality}, one can compute that 
\[
\zeta^{\dia}_N(\bk)=\sum_{A\subset [r]^1_{\bk}}\zeta^{}_N\left(\genfrac{}{}{0pt}{}{(\{1\}^r)_{(+A)}}{\bk_{(-A)}}\right)=\sum_{A\subset[r]^1_{\bk}}\zeta^{\flat}_N\left(\genfrac{}{}{0pt}{}{(\{1\}^r)_{(+A)}}{\bk_{(-A)}}\right)=\zeta^{\dia\flat}_N(\bk),
\]
where $r=\dep(\bk)$.
\end{proof}
In \cref{sec:Kawashima}, we prove that more relations are satisfied, but from this expression, it is clear that the duality $\zeta^{\dia}(\bk)=\zeta^{\dia}(\bk^{\dagger})$ holds via the transformation of the form $n\mapsto N-n$.

\section{Difference calculation}\label{sec:difference}
Throughout this section, we fix a positive integer $N$ and an even-depth index $\bc=(c_1,\dots, c_{2s})$.
Also, recall the notation explained after the definition of $\cD$ in \cref{subsec:1st-main}.
For a non-negative integer $n$, $[n]_0\coloneqq [n]\cup\{0\}$.
We define $f_N(\bc)$ and $g_N(\bc)$ by
\begin{align*}
f_N(\bc)&\coloneqq \zeta^{\dia}_N(\{1\}^{c_1-1},c_2+1,\dots, \{1\}^{c_{2s-1}-1},c_{2s}+1), \\
g_N(\bc)&\coloneqq f_{N+1}(\bc).
\end{align*}
For $\bn=(n_{i,j})_{i\in[2s],j\in[c_i]}$ and a subset $A\subset[2s]$, let
\[
P_{N,\bc}(\bn;A)\coloneqq\Biggl(\prod_{\substack{i\in[2s]\setminus A \\ i:\text{odd}}}\prod_{j\in[c_i]}\frac{1}{N-n_{i,j}}\Biggr)\Biggl(\prod_{\substack{i\in[2s]\setminus A \\ i:\text{even}}}\prod_{j\in[c_i]}\frac{1}{n_{i,j}}\Biggr).
\]
Then, by \cref{cor:dia-discrete_integral}, we have
\[
f_N(\bc)=\sum_{\substack{0<n_{i,1}\leq \cdots \leq n_{i,c_i}<N \ (i\in [2s]) \\ n_{i,c_i}\leq n_{i+1,1} \ (i\in[2s-1], \ i:\text{odd}) \\ n_{i,c_i}<n_{i+1,1} \ (i\in [2s-1], \ i:\text{even})}}P_{N,\bc}(\bn;\varnothing)
\]
and
\[
g_N(\bc)=\sum_{\substack{0\leq n_{i,1}\leq \cdots \leq n_{i,c_i}\leq N \ (i\in [2s]) \\ n_{i,c_i}\leq n_{i+1,1} \ (i\in[2s-1], \ i:\text{even}) \\ n_{i,c_i}<n_{i+1,1} \ (i\in [2s-1], \ i:\text{odd})}}P_{N,\bc}(\bn;\varnothing)
\]
(by replacing $n_{i,j}$ with $n_{i,j}+1$ only when $i$ is odd).
Furthermore, we define $h_N(\bc)$ by
\[
h_N(\bc)\coloneqq \sum_{A\subset[2s]^1_{\bc}}\frac{1}{N^{\#A}}\sum_{\substack{0\leq n_{i,1}\leq \cdots \leq n_{i,c_i}<N \ (i\in [2s]) \\ n_{i,c_i}<n_{i+1,1} \ (i\in [2s-1]_0, \ i+1\not\in A) \\ n_{i,c_i}=n_{i+1,1} \ (i\in [2s-1]_0, \ i+1\in A)}}P_{N,\bc}(\bn;A).
\]
When $i=0$, read $n_{i,c_i}$ as $0$.

\begin{lemma}\label{lem:f-h}
\begin{equation}\label{eq:f=h}
\sum_{\substack{A\subset[2s]^{>1}_{\bc} \\ \{2r-1,2r\}\not\subset A \ (r\in[s])}}\frac{(-1)^{\#A}}{N^{\#A}}f_N(\bc-\bdelta_A)=\sum_{\substack{A\subset [2s]^1_{\bc} \\ A:\textnormal{odd-even}}}\frac{(-1)^{\#A/2}}{N^{\#A}}h_N(\bc_{(-A)}).
\end{equation}
\end{lemma}
\begin{proof}
By the definition of $f_N$, we have
\begin{align*}
\text{LHS of }\eqref{eq:f=h}&=\sum_{\substack{A\subset[2s]^{>1}_{\bc} \\ \{2r-1,2r\}\not\subset A \ (r\in[s])}}(-1)^{\#A}\sum_{\substack{0<n_{i,1}\leq \cdots \leq n_{i,c_i}<N \ (i\in [2s]) \\ n_{i,c_i}\leq n_{i+1,1} \ (i\in[2s-1], \ i:\text{odd}, \ i\not\in A \text{ and } i+1\not\in A) \\ n_{i,c_i}=n_{i+1,1} \ (i\in[2s-1], \ i:\text{odd}, \ i\in A \text{ or } i+1\in A) \\ n_{i,c_i}<n_{i+1,1} \ (i\in [2s-1], \ i:\text{even})}}\\
&\qquad P_{N,\bc}(\bn;\varnothing)\Biggl(\prod_{\substack{i\in A \\ i:\text{odd}}}\frac{N-n_{i,c_i}}{N}\Biggr)\Biggl(\prod_{\substack{i\in A \\ i:\text{even}}}\frac{n_{i,1}}{N}\Biggr).
\end{align*}
By the correspondence $A\mapsto S\coloneqq\{r\in [s] \mid 2r-1 \in A \text{ or } 2r\in A\}$, we compute
\begin{align*}
\text{LHS of }\eqref{eq:f=h} &=\sum_{S\subset[s]}(-1)^{\#S}\sum_{\substack{0<n_{i,1}\leq \cdots \leq n_{i,c_i}<N \ (i\in [2s]) \\ n_{i,c_i}\leq n_{i+1,1} \ (i\in[2s-1], \ i:\text{odd}, \ \frac{i+1}{2}\not\in S) \\ n_{i,c_i}=n_{i+1,1} \ (i\in[2s-1], \ i:\text{odd}, \ \frac{i+1}{2}\in S) \\ n_{i,c_i}<n_{i+1,1} \ (i\in [2s-1], \ i:\text{even})}}P_{N,\bc}(\bn;\varnothing) \\
&\qquad \times\Biggl(\prod_{r\in S}\frac{(1-\delta_{1,c_{2r-1}})(N-n_{2r-1,c_{2r-1}})+(1-\delta_{1,c_{2r}})n_{2r,1}}{N}\Biggr) \\
&=\sum_{S\subset[s]}(-1)^{\#S}\sum_{S\subset T\subset [s]}\sum_{\substack{0<n_{i,1}\leq \cdots \leq n_{i,c_i}<N \ (i\in [2s]) \\ n_{i,c_i} < n_{i+1,1} \ (i\in[2s-1], \ i:\text{odd}, \ \frac{i+1}{2}\not\in T) \\ n_{i,c_i}=n_{i+1,1} \ (i\in[2s-1], \ i:\text{odd}, \ \frac{i+1}{2}\in T) \\ n_{i,c_i}<n_{i+1,1} \ (i\in [2s-1], \ i:\text{even})}}P_{N,\bc}(\bn;\varnothing) \\
&\qquad \times\Biggl(\prod_{r\in S}\frac{(1-\delta_{1,c_{2r-1}})(N-n_{2r-1,c_{2r-1}})+(1-\delta_{1,c_{2r}})n_{2r,1}}{N}\Biggr)\\
&=\sum_{T\subset[s]}\sum_{\substack{0<n_{i,1}\leq \cdots \leq n_{i,c_i}<N \ (i\in [2s]) \\ n_{i,c_i} < n_{i+1,1} \ (i\in[2s-1], \ i:\text{odd}, \ \frac{i+1}{2}\not\in T) \\ n_{i,c_i}=n_{i+1,1} \ (i\in[2s-1], \ i:\text{odd}, \ \frac{i+1}{2}\in T) \\ n_{i,c_i}<n_{i+1,1} \ (i\in [2s-1], \ i:\text{even})}}P_{N,\bc}(\bn;\varnothing) \\
&\qquad \times\sum_{S\subset T}(-1)^{\#S}\Biggl(\prod_{r\in S}\frac{(1-\delta_{1,c_{2r-1}})(N-n_{2r-1,c_{2r-1}})+(1-\delta_{1,c_{2r}})n_{2r,1}}{N}\Biggr).
\end{align*}
Here, $\delta_{i,j}$ is the Kronecker delta.
When $n_{2r-1,c_{2r-1}}=n_{2r,1}$ for every $r\in T$, we obtain
\begin{align*}
&\sum_{S\subset T}(-1)^{\#S}\Biggl(\prod_{r\in S}\frac{(1-\delta_{1,c_{2r-1}})(N-n_{2r-1,c_{2r-1}})+(1-\delta_{1,c_{2r}})n_{2r,1}}{N}\Biggr)\\
&=\sum_{r\in T}\left(1-\frac{(1-\delta_{1,c_{2r-1}})(N-n_{2r-1,c_{2r-1}})+(1-\delta_{1,c_{2r}})n_{2r,1}}{N}\right)\\
&=\prod_{r\in T}\frac{\delta_{1,c_{2r-1}}(N-n_{2r-1,c_{2r-1}})+\delta_{1,c_{2r}}n_{2r,1}}{N}
\end{align*}
and hence we can rewrite the sum as
\begin{align*}
&\text{LHS of }\eqref{eq:f=h} \\
&=\sum_{T\subset[s]}\sum_{\substack{0<n_{i,1}\leq \cdots \leq n_{i,c_i}<N \ (i\in [2s]) \\ n_{i,c_i} < n_{i+1,1} \ (i\in[2s-1], \ i:\text{odd}, \ \frac{i+1}{2}\not\in T) \\ n_{i,c_i}=n_{i+1,1} \ (i\in[2s-1], \ i:\text{odd}, \ \frac{i+1}{2}\in T) \\ n_{i,c_i}<n_{i+1,1} \ (i\in [2s-1], \ i:\text{even})}}P_{N,\bc}(\bn;\varnothing)\prod_{r\in T}\frac{\delta_{1,c_{2r-1}}(N-n_{2r-1,c_{2r-1}})+\delta_{1,c_{2r}}n_{2r,1}}{N}\\
&=\sum_{\substack{B\subset[2s]^{1}_{\bc} \\ \{2r-1,2r\}\not\subset B \ (r\in[s])}}\frac{1}{N^{\#B}}\sum_{\substack{0<n_{i,1}\leq \cdots \leq n_{i,c_i}<N \ (i\in [2s]) \\ n_{i,c_i} < n_{i+1,1} \ (i\in[2s-1], \ i:\text{odd}, \ i\not\in B \text{ and } i+1\not\in B) \\ n_{i,c_i}=n_{i+1,1} \ (i\in[2s-1], \ i:\text{odd}, \ i\in B \text{ or } i+1\in B) \\ n_{i,c_i}<n_{i+1,1} \ (i\in [2s-1], \ i:\text{even})}}P_{N,\bc}(\bn;B)\\
&=\sum_{\substack{B\subset[2s]^{1}_{\bc} \\ \{2r-1,2r\}\not\subset B \ (r\in[s])}}\frac{1}{N^{\#B}}\sum_{\substack{0\leq n_{i,1}\leq \cdots \leq n_{i,c_i}<N \ (i\in[2s]) \\ n_{i,c_i} < n_{i+1,1} \ (i\in[2s-1]_0, \ i+1\not\in B) \\ n_{i,c_i}=n_{i+1,1} \ (i\in[2s-1]_0, \ i+1\in B)}}P_{N,\bc}(\bn;B).
\end{align*}
In this step, we consider the correspondence $B\mapsto T\coloneqq\{r\in[s]\mid 2r-1\in B \text{ or } 2r\in B\}$; whenever $2r-1\in B$, we replace $n_{2r-2,c_{2r-2}}<n_{2r-1,1}=n_{2r,1}$ with $n_{2r-2,c_{2r-2}}=n_{2r-1,1}<n_{2r,1}$.

On the other hand, by the definition of $h_N$ and the correspondence $A\mapsto U\coloneqq \{r\in[s] \mid 2r-1\in A \text{ and } 2r\in A\}$,
\begin{align*}
\text{RHS of }\eqref{eq:f=h} &=\sum_{U\subset[s]}(-1)^{\#U}\sum_{\substack{B\subset[2s]_{\bc}^1 \\ \{2r-1,2r\}\subset B \ (r\in U)}}\frac{1}{N^{\#B}}\sum_{\substack{0\leq n_{i,1}\leq \cdots \leq n_{i,c_i} < N \ (i\in[2s]) \\ n_{i,c_i}<n_{i+1,1} \ (i\in[2s-1]_0, i+1\not\in B) \\ n_{i,c_i}=n_{i+1,1} \ (i\in[2s-1]_0, i+1\in B)}}P_{N,\bc}(\bn; B)\\
&=\sum_{B\subset[2s]_{\bc}^1}\Biggl(\sum_{\substack{U\subset[s] \\ \{2r-1,2r\}\subset B \ (r\in U)}}(-1)^{\#U}\Biggr)\frac{1}{N^{\#B}}\sum_{\substack{0\leq n_{i,1}\leq \cdots \leq n_{i,c_i} < N \ (i\in[2s]) \\ n_{i,c_i}<n_{i+1,1} \ (i\in[2s-1]_0, i+1\not\in B) \\ n_{i,c_i}=n_{i+1,1} \ (i\in[2s-1]_0, i+1\in B)}}P_{N,\bc}(\bn; B).
\end{align*}
Since
\begin{align*}
\sum_{\substack{U\subset[s] \\ \{2r-1,2r\}\subset B \ (r\in U)}}(-1)^{\#U}=\begin{cases} 1 & \text{if } \{2r-1,2r\}\not\subset B \text{ for all } r\in[s], \\ 0 & \text{otherwise}\end{cases}
\end{align*}
holds, we have
\[
\text{RHS of }\eqref{eq:f=h}=\sum_{\substack{B\subset[2s]^{1}_{\bc} \\ \{2r-1,2r\}\not\subset B \ (r\in[s])}}\frac{1}{N^{\#B}}\sum_{\substack{0\leq n_{i,1}\leq \cdots \leq n_{i,c_i}<N \ (i\in[2s]) \\ n_{i,c_i} < n_{i+1,1} \ (i\in[2s-1]_0, \ i+1\not\in B) \\ n_{i,c_i}=n_{i+1,1} \ (i\in[2s-1]_0, \ i+1\in B)}}P_{N,\bc}(\bn;B).
\]
Combined with the result of the previous calculation, this completes the proof.
\end{proof}
\begin{corollary}\label{cor:h-f}
\begin{equation}\label{eq:h-f}
h_N(\bc)=\sum_{\substack{A\subset[2s]^1_{\bc} \\ A: \textnormal{odd-even}}}\sum_{\substack{B\subset[2s]^{>1}_{\bc} \\ \{2r-1,2r\}\not\subset B \ (r\in[s])}}\frac{(-1)^{\#B}}{N^{\#A+\#B}}f_N(\bc_{(-A)}-\bdelta_B).
\end{equation}
\end{corollary}
\begin{proof}
By \cref{lem:f-h}, we compute
\begin{align*}
\text{RHS of } \eqref{eq:h-f}&=\sum_{\substack{A\subset [2s]^1_{\bc} \\ A:\text{odd-even}}}\frac{1}{N^{\#A}}\sum_{\substack{B\subset[2s]^1_{\bc}\setminus A \\ B:\text{odd-even}}}\frac{(-1)^{\frac{\#B}{2}}}{N^{\#B}}h_N(\bc_{-(A\cup B)})\\
&=\sum_{\substack{A''\subset[2s]^1_{\bc} \\ A'':\text{odd-even}}}\Biggl(\sum_{\substack{A\subset A'' \\ A:\text{odd-even}}}(-1)^{\frac{\#A}{2}}\Biggr)\frac{(-1)^{\frac{\#A''}{2}}}{N^{\#A''}}h_N(\bc_{(-A'')}).
\end{align*}
Since
\[
\sum_{\substack{A\subset A'' \\ A: \text{odd-even}}}(-1)^{\#A/2}=0
\]
holds when $A''$ is not empty, we have the conclusion.
\end{proof}
\begin{lemma}\label{lem:g-h}
\begin{equation}\label{eq:g=h}
\sum_{\substack{A\subset[2s]^{>1}_{\bc} \\ \{2r,2r+1\}\not\subset A \ (r\in[s-1])}}\frac{(-1)^{\#A}}{N^{\#A}}g_N(\bc-\bdelta_A)=\sum_{\substack{A\subset [2s]^1_{\bc} \\ A:\textnormal{even-odd}}}\frac{(-1)^{\#A/2}}{N^{\#A}}h_N(\bc_{(-A)}).
\end{equation}
\end{lemma}
\begin{proof}
The proof is similar to that of \cref{lem:f-h}, but because there are differences in the details, we include it here for the reader’s convenience.
By the definition of $g_N$, we have
\begin{align*}
\text{LHS of } \eqref{eq:g=h} &=\sum_{\substack{A\subset[2s]^{>1}_{\bc} \\ \{2r,2r+1\}\not\subset A \ (r\in[s-1])}}(-1)^{\#A}\sum_{\substack{0\leq n_{i,1}\leq \cdots \leq n_{i,c_i}\leq N \ (i\in [2s]) \\ n_{i,c_i}\leq n_{i+1,1} \ (i\in[2s]_0, \ i:\text{even}, \ i\not\in A \text{ and } i+1\not\in A) \\ n_{i,c_i}=n_{i+1,1} \ (i\in[2s]_0, \ i:\text{even}, \ i\in A \text{ or } i+1\in A) \\ n_{i,c_i}<n_{i+1,1} \ (i\in [2s-1], \ i:\text{odd})}}\\
&\qquad P_{N,\bc}(\bn;\varnothing)\Biggl(\prod_{\substack{i\in A \\ i:\text{odd}}}\frac{N-n_{i,1}}{N}\Biggr)\Biggl(\prod_{\substack{i\in A \\ i:\text{even}}}\frac{n_{i,c_i}}{N}\Biggr).
\end{align*}
When $i=2s$, read $n_{i+1,1}$ as $N$.
Let $c_0=c_{2s+1}=1$.
By the correspondence $A\mapsto S\coloneqq\{r\in [s]_0 \mid 2r \in A \text{ or } 2r+1\in A\}$, we compute
\begin{align*}
\text{LHS of } \eqref{eq:g=h} &=\sum_{S\subset[s]_0}(-1)^{\#S}\sum_{\substack{0\leq n_{i,1}\leq \cdots \leq n_{i,c_i}\leq N \ (i\in [2s]) \\ n_{i,c_i}\leq n_{i+1,1} \ (i\in[2s]_0, \ i:\text{even}, \ \frac{i}{2}\not\in S) \\ n_{i,c_i}=n_{i+1,1} \ (i\in[2s]_0, \ i:\text{even}, \ \frac{i}{2}\in S) \\ n_{i,c_i}<n_{i+1,1} \ (i\in [2s-1], \ i:\text{odd})}}P_{N,\bc}(\bn;\varnothing) \\
&\qquad \times\Biggl(\prod_{r\in S}\frac{(1-\delta_{1,c_{2r}})n_{2r,c_{2r}}+(1-\delta_{1,c_{2r+1}})(N-n_{2r+1,1})}{N}\Biggr) \\
&=\sum_{S\subset[s]_0}(-1)^{\#S}\sum_{S\subset T\subset [s]_0}\sum_{\substack{0\leq n_{i,1}\leq \cdots \leq n_{i,c_i}\leq N \ (i\in [2s]) \\ n_{i,c_i} < n_{i+1,1} \ (i\in[2s]_0, \ i:\text{even}, \ \frac{i}{2}\not\in T) \\ n_{i,c_i}=n_{i+1,1} \ (i\in[2s]_0, \ i:\text{even}, \ \frac{i}{2}\in T) \\ n_{i,c_i}<n_{i+1,1} \ (i\in [2s-1], \ i:\text{odd})}}P_{N,\bc}(\bn;\varnothing) \\
&\qquad \times\Biggl(\prod_{r\in S}\frac{(1-\delta_{1,c_{2r}})n_{2r,c_{2r}}+(1-\delta_{1,c_{2r+1}})(N-n_{2r+1,1})}{N}\Biggr)\\
&=\sum_{T\subset[s]_0}\sum_{\substack{0\leq n_{i,1}\leq \cdots \leq n_{i,c_i}\leq N \ (i\in [2s]) \\ n_{i,c_i} < n_{i+1,1} \ (i\in[2s]_0, \ i:\text{even}, \ \frac{i}{2}\not\in T) \\ n_{i,c_i}=n_{i+1,1} \ (i\in[2s]_0, \ i:\text{even}, \ \frac{i}{2}\in T) \\ n_{i,c_i}<n_{i+1,1} \ (i\in [2s-1], \ i:\text{odd})}}P_{N,\bc}(\bn;\varnothing) \\
&\qquad \times\sum_{S\subset T}(-1)^{\#S}\Biggl(\prod_{r\in S}\frac{(1-\delta_{1,c_{2r}})n_{2r,c_{2r}}+(1-\delta_{1,c_{2r+1}})(N-n_{2r+1,1})}{N}\Biggr).
\end{align*}
When $n_{2r,c_{2r}}=n_{2r+1,1}$ for every $r\in T$, we obtain
\begin{align*}
&\sum_{S\subset T}(-1)^{\#S}\Biggl(\prod_{r\in S}\frac{(1-\delta_{1,c_{2r}})n_{2r,c_{2r}}+(1-\delta_{1,c_{2r+1}})(N-n_{2r+1,1})}{N}\Biggr)\\
&=\sum_{r\in T}\left(1-\frac{(1-\delta_{1,c_{2r}})n_{2r,c_{2r}}+(1-\delta_{1,c_{2r+1}})(N-n_{2r+1,1})}{N}\right)\\
&=\prod_{r\in T}\frac{\delta_{1,c_{2r}}n_{2r,c_{2r}}+\delta_{1,c_{2r+1}}(N-n_{2r+1,1})}{N}
\end{align*}
and hence we can rewrite the sum as
\begin{align*}
&\text{LHS of } \eqref{eq:g=h} \\
&=\sum_{T\subset[s]_0}\sum_{\substack{0\leq n_{i,1}\leq \cdots \leq n_{i,c_i}\leq N \ (i\in [2s]) \\ n_{i,c_i} < n_{i+1,1} \ (i\in[2s]_0, \ i:\text{even}, \ \frac{i}{2}\not\in T) \\ n_{i,c_i}=n_{i+1,1} \ (i\in[2s]_0, \ i:\text{even}, \ \frac{i}{2}\in T) \\ n_{i,c_i}<n_{i+1,1} \ (i\in [2s-1], \ i:\text{odd})}}P_{N,\bc}(\bn;\varnothing)\prod_{r\in T}\frac{\delta_{1,c_{2r}}n_{2r,c_{2r}}+\delta_{1,c_{2r+1}}(N-n_{2r+1,1})}{N}\\
&=\sum_{\substack{B\subset[2s]^{1}_{\bc} \\ \{2r,2r+1\}\not\subset B \ (r\in[s-1])}}\frac{1}{N^{\#B}}\sum_{\substack{0\leq n_{i,1}\leq \cdots \leq n_{i,c_i}\leq N \ (i\in [2s]) \\ n_{i,c_i} < n_{i+1,1} \ (i\in[2s]_0, \ i:\text{even}, \ i\not\in B \text{ and } i+1\not\in B) \\ n_{i,c_i}=n_{i+1,1} \ (i\in[2s]_0, \ i:\text{even}, \ i\in B \text{ or } i+1\in B) \\ n_{i,c_i}<n_{i+1,1} \ (i\in [2s-1], \ i:\text{odd})}}P_{N,\bc}(\bn;B)\\
&=\sum_{\substack{B\subset[2s]^{1}_{\bc} \\ \{2r,2r+1\}\not\subset B \ (r\in[s-1])}}\frac{1}{N^{\#B}}\sum_{\substack{0\leq n_{i,1}\leq \cdots \leq n_{i,c_i}< N \ (i\in[2s]) \\ n_{i,c_i} < n_{i+1,1} \ (i\in[2s-1]_0, \ i+1\not\in B) \\ n_{i,c_i}=n_{i+1,1} \ (i\in[2s-1]_0, \ i+1\in B)}}P_{N,\bc}(\bn;B).
\end{align*}
In this step, we consider the correspondence $B\mapsto T\coloneqq\{r\in[s]_0\mid 2r\in B \text{ or } 2r+1\in B\}$; whenever $2r\in B$, we replace $n_{2r-1,c_{2r-1}}<n_{2r,1}=n_{2r+1,1}$ with $n_{2r-1,c_{2r-1}}=n_{2r,1}<n_{2r+1,1}$.

On the other hand, by the definition of $h_N$ and the correspondence $A\mapsto U\coloneqq \{r\in[s-1] \mid 2r\in A \text{ and } 2r+1\in A\}$,
\begin{align*}
\text{RHS of } \eqref{eq:g=h}&=\sum_{U\subset[s-1]}(-1)^{\#U}\sum_{\substack{B\subset[2s]_{\bc}^1 \\ \{2r,2r+1\}\subset B \ (r\in U)}}\frac{1}{N^{\#B}}\sum_{\substack{0\leq n_{i,1}\leq \cdots \leq n_{i,c_i} < N \ (i\in[2s]) \\ n_{i,c_i}<n_{i+1,1} \ (i\in[2s-1]_0, i+1\not\in B) \\ n_{i,c_i}=n_{i+1,1} \ (i\in[2s-1]_0, i+1\in B)}}P_{N,\bc}(\bn;B)\\
&=\sum_{B\subset[2s]_{\bc}^1}\Biggl(\sum_{\substack{U\subset[s-1] \\ \{2r,2r+1\}\subset B \ (r\in U)}}(-1)^{\#U}\Biggr)\frac{1}{N^{\#B}}\sum_{\substack{0\leq n_{i,1}\leq \cdots \leq n_{i,c_i} < N \ (i\in[2s]) \\ n_{i,c_i}<n_{i+1,1} \ (i\in[2s-1]_0, i+1\not\in B) \\ n_{i,c_i}=n_{i+1,1} \ (i\in[2s-1]_0, i+1\in B)}}P_{N,\bc}(\bn;B).
\end{align*}
Since
\begin{align*}
\sum_{\substack{U\subset[s-1] \\ \{2r,2r+1\}\subset B \ (r\in U)}}(-1)^{\#U}=\begin{cases} 1 & \text{if } \{2r,2r+1\}\not\subset B \text{ for all } r\in[s-1], \\ 0 & \text{otherwise}\end{cases}
\end{align*}
holds, we have
\[
\text{RHS of } \eqref{eq:g=h}=\sum_{\substack{B\subset[2s]^{1}_{\bc} \\ \{2r,2r+1\}\not\subset B \ (r\in[s-1])}}\frac{1}{N^{\#B}}\sum_{\substack{0\leq n_{i,1}\leq \cdots \leq n_{i,c_i}<N \ (i\in[2s]) \\ n_{i,c_i} < n_{i+1,1} \ (i\in[2s-1]_0, \ i+1\not\in B) \\ n_{i,c_i}=n_{i+1,1} \ (i\in[2s-1]_0, \ i+1\in B)}}P_{N,\bc}(\bn;B).
\]
Combined with the result of the previous calculation, this completes the proof.
\end{proof}
\begin{corollary}\label{cor:h-g}
\[
h_N(\bc)=\sum_{\substack{A\subset[2s]^1_{\bc} \\ A: \textnormal{even-odd}}}\sum_{\substack{B\subset[2s]^{>1}_{\bc} \\ \{2r,2r+1\}\not\subset B \ (r\in[s-1])}}\frac{(-1)^{\#B}}{N^{\#A+\#B}}g_N(\bc_{(-A)}-\bdelta_B).
\]
\end{corollary}
\begin{proof}
Using \cref{lem:g-h}, the proof is similar to the proof of \cref{cor:h-f}.
\end{proof}
\begin{theorem}[Difference equation]\label{thm:difference-eq}
Let $(\Delta f)_N(\bc)$ denote the difference $f_{N+1}(\bc)-f_N(\bc)$ $(=g_N(\bc)-f_N(\bc))$.
Then, we have
\begin{align*}
(\Delta f)_N(\bc)&=\sum_{\substack{A\subset[2s]^1_{\bc} \\ A: \textnormal{even-odd}}}\sum_{\substack{B\subset[2s]^{>1}_{\bc} \\ \#A+\#B\geq 1 \\ \{2r,2r+1\}\not\subset B \ (r\in[s-1])}}\frac{(-1)^{\#B-1}}{N^{\#A+\#B}}(\Delta f)_N(\bc_{(-A)}-\bdelta_B)\\
&\quad +\sum_{\substack{A\subset[2s]^1_{\bc} \\ A: \textnormal{even-odd}}}\sum_{\substack{B\subset[2s]^{>1}_{\bc} \\ \#A+\#B\geq 2 \\ \{2r,2r+1\}\not\subset B \ (r\in[s-1])}}\frac{(-1)^{\#B-1}}{N^{\#A+\#B}}f_N(\bc_{(-A)}-\bdelta_B)\\
&\quad +\sum_{\substack{A\subset[2s]^1_{\bc} \\ A: \textnormal{odd-even}}}\sum_{\substack{B\subset[2s]^{>1}_{\bc} \\ \#A+\#B\geq 2 \\ \{2r-1,2r\}\not\subset B \ (r\in[s])}}\frac{(-1)^{\#B}}{N^{\#A+\#B}}f_N(\bc_{(-A)}-\bdelta_B).
\end{align*}
\end{theorem}
\begin{proof}
By combining \cref{cor:h-f} and \cref{cor:h-g}, we have
\begin{align*}
&\sum_{\substack{A\subset[2s]^1_{\bc} \\ A: \text{even-odd}}}\sum_{\substack{B\subset[2s]^{>1}_{\bc} \\ \{2r,2r+1\}\not\subset B \ (r\in[s-1])}}\frac{(-1)^{\#B}}{N^{\#A+\#B}}g_N(\bc_{(-A)}-\bdelta_B)\\
&=\sum_{\substack{A\subset[2s]^1_{\bc} \\ A: \text{odd-even}}}\sum_{\substack{B\subset[2s]^{>1}_{\bc} \\ \{2r-1,2r\}\not\subset B \ (r\in[s])}}\frac{(-1)^{\#B}}{N^{\#A+\#B}}f_N(\bc_{(-A)}-\bdelta_B).
\end{align*}
Noting that
\[
\sum_{\substack{B\subset[2s]^{>1}_{\bc}, \ \#B=1 \\ \{2r,2r+1\}\not\subset B \ (r\in[s-1])}}\frac{(-1)^{\#B}}{N^{\#B}}f_N(\bc-\bdelta_B)=\sum_{\substack{B\subset[2s]^{>1}_{\bc}, \#B=1 \\ \{2r-1,2r\}\not\subset B \ (r\in[s])}}\frac{(-1)^{\#B}}{N^{\#B}}f_N(\bc-\bdelta_B),
\]
we obtain the formula from the above equation.
\end{proof}
\begin{proof}[Proof of $\cref{thm:dia-main2}$]
Note that $\zeta^{\dia}_N(\bk)=\zeta^{}_N(\bk)$ holds for each $\bk\in\II^{\geq 2}_k$, and that
\[
\sum_{n=1}^{N-1}\frac{1}{n^a}\zeta^{}_{n}(k_1,\dots, k_r)=\zeta^{}_N(k_1,\dots, k_r, a)
\]
and
\[
\sum_{n=1}^{N-1}\frac{1}{n^a}\bigl(\zeta^{}_{n+1}(k_1,\dots, k_r)-\zeta^{}_{n}(k_1,\dots, k_r)\bigr)=\zeta^{}_N(k_1,\dots, k_{r-1}, k_r+a)
\]
hold for multiple harmonic sums.
Then, by induction on the weight of $\bc$, it suffices to take the sum from $n=1$ to $N-1$ of the difference relations given in \cref{thm:difference-eq}.
\end{proof}

\section{\texorpdfstring{Kawashima's relations for $Z^{\dia}$}{Kawashima's relations for dia-values}}\label{sec:Kawashima}
Throughout this section, we fix a positive integer $N$ and an indeterminate $t$.
\subsection{Kawashima-like function}
\begin{definition}
For an index $\bk=(k_1,\dots, k_r)$, we define $F_N(\bk;t)$ as
\[
F_N(\bk;t)\coloneqq\sum_{A\subset [r]}(-1)^{\#A}\sum_{(n_1,\dots, n_r)\in \overline{S}_{r,N}(A)}\Biggl(\prod_{i\in A}\frac{1}{(N-n_i+t)^{k_i}}\Biggr)\Biggl(\prod_{i\in[r]\setminus A}\frac{1}{n_i^{k_i}}\Biggr),
\]
where
\[
\overline{S}_{r,N}(A) \coloneqq \left\{(n_1,\dots,n_r) \in [N-1]^r \ \middle| \ \begin{array}{cc} 
n_{i-1}\leq n_{i}  & \text{ if } 1<i\in [r]\setminus A, \\ 
n_{i-1} < n_{i} & \text{ if } 1<i\in A
\end{array} \right\}
.
\]
\end{definition}
The star-harmonic product $\overline{*}$ on $\cH^1$ is defined by $w \, \overline{*} \, 1=1 \, \overline{*} \, w=w$ for any word $w\in\cH^1$ and $w_1e_{k_1} \, \overline{*} \, w_2e_{k_2}=(w_1 \, \overline{*} \, w_2e_{k_2})e_{k_1}+(w_1e_{k_1} \, \overline{*} \, w_2)e_{k_2}-(w_1 \, \overline{*} \, w_2)e_{k_1+k_2}$ for any words $w_1$, $w_2\in\cH^1$ and any positive integers $k_1$, $k_2$ with $\QQ$-bilinearity.
Let $\cF_N\colon\cH^1\to\QQ(t)$ be the $\QQ$-linear mapping defined by $\cF_N(e_{k_1}\cdots e_{k_r})=F_N(\bk;t)$ for each index $\bk=(k_1,\dots, k_r)$ and $\cF_N(1)=1$.
\begin{proposition}\label{prop:F-harmonic}
For any $w_1, w_2\in\cH^1$, we have
\[
\cF_N(w_1;t)\cF_N(w_2;t)=\cF_N(w_1 \, \overline{*} \, w_2;t).
\]
\end{proposition}
\begin{proof}
Let $\bk=(k_1,\dots, k_r)$ and $\bl=(k_{r+1},\dots, k_{r+s})$ be indices.
Recall the definition of the set $\Sigma_{r,s,d}$ for each $d\in[\min\{r,s\}]_0$ and the notation $(k_1^{\sigma},\dots, k_{r+s-d}^{\sigma})$ for $\sigma \in \Sigma_{r,s,d}$.
\begin{align*}
&F_N(\bk;t)F_N(\bl;t)\\
&=\Biggl(\sum_{A\subset[r]}(-1)^{\#A}\sum_{(n_1,\dots, n_r)\in\overline{S}_{r,N}(A)}\Biggl(\prod_{i\in A}\frac{1}{(N-n_i+t)^{k_i}}\Biggr)\Biggl(\prod_{i\in[r]\setminus A}\frac{1}{n_i^{k_i}}\Biggr)\Biggr)\\
&\quad\times\Biggl(\sum_{B\subset[r+s]\setminus[r]}(-1)^{\#B}\sum_{(n_{r+1},\dots, n_{r+s})\in\overline{S}_{s,N}(B-r)}\Biggl(\prod_{i\in B}\frac{1}{(N-n_i+t)^{k_i}}\Biggr)\Biggl(\prod_{i\in[r+s]\setminus ([r]\cup B)}\frac{1}{n_i^{k_i}}\Biggr)\Biggr)\\
&=\sum_{A\subset[r]}\sum_{B\subset[r+s]\setminus[r]}(-1)^{\#A+\#B}\sum_{d=0}^{\min\{r,s\}}(-1)^{d-\#\{j\in\sigma(A\cup B) \mid\#\sigma^{-1}(\{j\})=2\}}\\
&\quad\sum_{\substack{\sigma\in\Sigma_{r,s,d} \\ \sigma^{-1}(\{j\})\subset A\cup B \text{ or } \sigma^{-1}(\{j\})\subset[r+s]\setminus (A\cup B) \ (j\in [r+s-d])}}\\
&\quad\sum_{(m_1,\dots, m_{r+s-d})\in\overline{S}_{r+s-d}(\sigma(A\cup B))}\Biggl(\prod_{j\in \sigma(A\cup B)}\frac{1}{(N-m_j+t)^{k_j^{\sigma}}}\Biggr)\Biggl(\prod_{j\in[r+s-d]\setminus \sigma(A\cup B)}\frac{1}{m_j^{k_j^{\sigma}}}\Biggr)\\
&=\sum_{d=0}^{\min\{r,s\}}(-1)^d\sum_{\sigma\in\Sigma_{r,s,d}}\sum_{U\subset [r+s-d]}(-1)^{\#U}\\
&\quad\sum_{(m_1,\dots, m_{r+s-d})\in\overline{S}_{r+s-d}(U)}\Biggl(\prod_{j\in U}\frac{1}{(N-m_j+t)^{k_j^{\sigma}}}\Biggr)\Biggl(\prod_{j\in[r+s-d]\setminus U}\frac{1}{m_j^{k_j^{\sigma}}}\Biggr)\\
&=\sum_{d=0}^{\min\{r,s\}}(-1)^d\sum_{\sigma\in\Sigma_{r,s,d}}F_N(k_1^{\sigma},\dots, k_{r+s-d}^{\sigma};t).
\end{align*}
Here, $B-r\coloneqq\{i-r \mid i\in B\}$.
This proves the desired statement.
\end{proof}
\begin{definition}
For an index $\bk=(k_1,\dots, k_r)$, we define $G_N(\bk;t)$ as
\begin{align*}
G_N(\bk;t)&\coloneqq\sum_{A\subset[r]_{\bk}^1\setminus\{r\}}\sum_{(n_1,\dots,n_r)\in S^{\star}_{r,N}(A)}\Biggl(\prod_{i\in A}\frac{1}{N-n_i+t}\Biggr)\Biggl(\prod_{i\in[r]\setminus A}\frac{1}{n_i^{k_i}}\Biggr)\\
&\qquad\times (-1)^{n_r-1}\binom{t}{n_r}\frac{(1-N)_{n_r}}{(1-N-t)_{n_r}}.
\end{align*}
Here, recall the definition of $S_{r,N}^{\star}(A)$ in \cref{prop:star-expression}.
The notation $(a)_n$ is the rising factorial, that is, $(a)_n=a(a+1)\cdots (a+n-1)$.
The binomial coefficient $\binom{t}{n}$ is the usual one, that is, $\binom{t}{n}=\frac{t(t-1)\cdots (t-n+1)}{n!}$.
\end{definition}
Let $Z^{\dia}_N\colon\cH^0\to\QQ$ be the $\QQ$-linear mapping defined by $Z^{\dia}_N(e_{k_1}\cdots e_{k_r})=\zeta^{\dia}_N(\bk)$ for each admissible index $\bk=(k_1,\dots, k_r)$ and $Z^{\dia}_N(1)=1$.
\begin{proposition}\label{prop:Taylor-G}
For an index $\bk=(k_1,\dots, k_r)$, we have
\[
G_N(\bk;t)=\sum_{m=1}^{\infty}(-1)^{m-1}Z^{\dia}_N(S(e_{k_1}\cdots e_{k_r})\circledast y^m)
t^m.
\]
Here, $S(e_{k_1}\cdots e_{k_r})\coloneqq\sum_{(l_1,\dots, l_s)\preceq\bk}e_{l_1}\cdots e_{l_s}$.
\end{proposition}
To prove this proposition, we first prepare two lemmas.
\begin{definition}
Let $z$ be an indeterminate and let $n\in[N-1]_0$.
Let $\bk=(k_1,\dots, k_r)$ be an index.
We define $\zeta^{\dia}_{n,N}(\bk;z)$ as
\[
\zeta^{\dia}_{n,N}(\bk;z)\coloneqq\sum_{A\subset[r]^1_{\bk}}\sum_{\substack{(n_1,\dots, n_{r+1})\in([N]\setminus[n])^{r+1} \\ n_i\leq n_{i+1} \ (i\in A) \\ n_i<n_{i+1} \ (i\in[r]\setminus A) \\ n_{r+1}=N}}\Biggl(\prod_{i\in A}\frac{1}{z-n_i}\Biggr)\Biggl(\prod_
{i\in[r]\setminus A}\frac{1}{n_i^{k_i}}\Biggr).
\]
Set $\zeta^{\dia}_N(\bk;z)\coloneqq\zeta_{0,N}^{\dia}(\bk;z)$.
Let $Z^{\dia}_{N,z}\colon\cH^1\to\QQ(z)$ be the $\QQ$-linear mapping defined by $Z^{\dia}_{N,z}(1)=1$ and $Z^{\dia}_{N,z}(e_{k_1}\cdots e_{k_r})=\zeta^{\dia}_N(\bk;z)$ for each index $\bk=(k_1,\dots, k_r)$.
\end{definition}
\begin{lemma}\label{lem:1...1-expansion}
For $n\in[N-1]_0$, we have
\[
\prod_{n<j<N}\left(1+\frac{t}{j}\right)\prod_{n<j\leq N}\left(1-\frac{t}{z-j}\right)^{-1}=\sum_{p=0}^{\infty}\zeta_{n,N}^{\dia}(\{1\}^p;z)t^p,
\]
where $\zeta_{n,N}^{\dia}(\{1\}^0;z)\coloneqq 1$.
\end{lemma}
\begin{proof}
This follows from expansions
\[
\prod_{n<j<N}\left(1+\frac{t}{j}\right)=\sum_{k=0}^{N-n-1}\sum_{n<a_1<\cdots<a_k<N}\frac{t^k}{a_1\cdots a_k}=\sum_{k=0}^{\infty}\sum_{n<a_1<\cdots<a_k<N}\frac{t^k}{a_1\cdots a_k}
\]
and
\[
\prod_{n<j\leq N}\left(1-\frac{t}{z-j}\right)^{-1}=\prod_{n<j\leq N}\sum_{l_j=0}^{\infty}\frac{t^{l_j}}{(z-j)^{l_j}}=\sum_{l=0}^{\infty}\sum_{n<b_1\leq \cdots \leq b_l\leq N}\frac{t^l}{(z-b_1)\cdots (z-b_l)},
\]
together with
\[
\sum_{k+l=p}\Biggl(\sum_{n<a_1<\cdots<a_k<N}\frac{1}{a_1\cdots a_k}\Biggr)\Biggl(\sum_{n<b_1\leq \cdots \leq b_l\leq N}\frac{1}{(z-b_1)\cdots (z-b_l)}\Biggr)=\zeta_{n,N}^{\dia}(\{1\}^p;z).\qedhere
\]
\end{proof}
\begin{lemma}\label{lem:induction-expansion}
For an index $\bk=(k_1,\dots, k_r)$, we have
\[
\zeta_N^{\dia}(\bk;z-t)\frac{(1+t)_{N-1}}{(N-1)!}\frac{(1-z)_N}{(1-z+t)_N}=\sum_{m=0}^{\infty}Z^{\dia}_{N,z}(e_{k_1}\cdots e_{k_r}\ast y^m)t^m.
\]
\end{lemma}
\begin{proof}
We prove this by induction on $r=\dep(\bk)$.
The initial step corresponds to the case where $\bk$ is the empty index, that is, the case $n=0$ in \cref{lem:1...1-expansion}.
Assuming that the statement holds when the depth is $r$, we can compute as follows using the induction hypothesis and \cref{lem:1...1-expansion}.
When $k_{r+1}\geq 2$, we compute
\begin{align*}
&\zeta_N^{\dia}(k_1,\dots, k_r, k_{r+1};z-t)\frac{(1+t)_{N-1}}{(N-1)!}\frac{(1-z)_N}{(1-z+t)_N}\\
&=\sum_{n=1}^{N-1}\frac{1}{n^{k_{r+1}}}\zeta_n^{\dia}(\bk;z-t)\frac{(1+t)_{n-1}}{(n-1)!}\frac{(1-z)_n}{(1-z+t)_n}\left(1+\frac{t}{n}\right)\\
&\qquad\times\prod_{n<j<N}\left(1+\frac{t}{j}\right)\prod_{n<j\leq N}\left(1-\frac{t}{z-j}\right)^{-1}\\
&=\sum_{n=1}^{N-1}\frac{1}{n^{k_{r+1}}}\Biggl(\sum_{q=0}^{\infty}Z^{\dia}_{n,z}(e_{k_1}\cdots e_{k_r}\ast y^q)t^q\Biggr)\left(1+\frac{t}{n}\right)\Biggl(\sum_{p=0}^{\infty}\zeta_{n,N}^{\dia}(\{1\}^p;z)t^p\Biggr)\\
&=\sum_{p=0}^{\infty}\sum_{q=0}^{\infty}Z_{N,z}^{\dia}((e_{k_1}\cdots e_{k_r}\ast y^q)e_{k_{r+1}}y^p)t^{p+q}+\sum_{p=1}^{\infty}\sum_{q=0}^{\infty}Z_{N,z}^{\dia}((e_{k_1}\cdots e_{k_r}\ast y^q)e_{k_{r+1}+1}y^{p-1})t^{p+q}.
\end{align*}
When $k_{r+1}=1$, we compute
\begin{align*}
&\zeta_N^{\dia}(k_1,\dots, k_r, 1;z-t)\frac{(1+t)_{N-1}}{(N-1)!}\frac{(1-z)_N}{(1-z+t)_N}\\
&=\sum_{n=1}^{N-1}\frac{1}{n}\zeta_n^{\dia}(\bk;z-t)\frac{(1+t)_{n-1}}{(n-1)!}\frac{(1-z)_n}{(1-z+t)_n}\left(1+\frac{t}{n}\right)\\
&\qquad\times\prod_{n<j<N}\left(1+\frac{t}{j}\right)\prod_{n<j\leq N}\left(1-\frac{t}{z-j}\right)^{-1}\\
&\quad+\sum_{n=1}^N\frac{1}{z-t-n}\zeta_n^{\dia}(\bk;z-t)\frac{(1+t)_{n-1}}{(n-1)!}\frac{(1-z)_n}{(1-z+t)_n}\left(1-\frac{t}{z-n}\right)\\
&\qquad\times\prod_{n\leq j<N}\left(1+\frac{t}{j}\right)\prod_{n\leq j\leq N}\left(1-\frac{t}{z-j}\right)^{-1}\\
&=\sum_{n=1}^{N-1}\frac{1}{n}\Biggl(\sum_{q=0}^{\infty}Z_{n,z}^{\dia}(e_{k_1}\cdots e_{k_r}\ast y^q)t^q\Biggr)\left(1+\frac{t}{n}\right)\Biggl(\sum_{p=0}^{\infty}\zeta_{n,N}^{\dia}(\{1\}^p;z)t^p\Biggr)\\
&\quad +\sum_{n=1}^{N}\frac{1}{z-n}\Biggl(\sum_{q=0}^{\infty}Z_{n,z}^{\dia}(e_{k_1}\cdots e_{k_r}\ast y^q)t^q\Biggr)\Biggl(\sum_{p=0}^{\infty}\zeta_{n-1,N}^{\dia}(\{1\}^p;z)t^p\Biggr)\\
&=\sum_{p=0}^{\infty}\sum_{q=0}^{\infty}Z_{N,z}^{\dia}((e_{k_1}\cdots e_{k_r}\ast y^q)y^{p+1})t^{p+q}+\sum_{p=1}^{\infty}\sum_{q=0}^{\infty}Z_{N,z}^{\dia}((e_{k_1}\cdots e_{k_r}\ast y^q)e_2y^{p-1})t^{p+q}.
\end{align*}
Since
\[
e_{k_1}\cdots e_{k_r}e_{k_{r+1}}\ast y^m=\sum_{\substack{p+q=m \\ p\geq 0, q\geq 0}}(e_{k_1}\cdots e_{k_r}\ast y^q)e_{k_{r+1}}y^p+\sum_{\substack{p+q=m \\ p\geq 1, q\geq 0}}(e_{k_1}\cdots e_{k_r}\ast y^q)e_{k_{r+1}+1}y^{p-1},
\]
the proof is complete.
\end{proof}
\begin{proof}[Proof of $\cref{prop:Taylor-G}$]
By an argument similar to the proof of \cref{prop:star-expression}, it suffices to show that the following identity holds:
\begin{align*}
&\sum_{A\subset[r]_{\bk}^1\setminus\{r\}}\sum_{(n_1,\dots,n_r)\in S_{r,N}(A)}\Biggl(\prod_{i\in A}\frac{1}{N-n_i+t}\Biggr)\Biggl(\prod_{i\in[r]\setminus A}\frac{1}{n_i^{k_i}}\Biggr)(-1)^{n_r}\binom{t}{n_r}\frac{(1-N)_{n_r}}{(1-N-t)_{n_r}}\\
&=\sum_{m=1}^{\infty}Z^{\dia}_N(e_{k_1}\cdots e_{k_r}\circledast y^m)(-t)^m.
\end{align*}
Since the left-hand side equals
\[
\sum_{n_r=1}^{N-1}\frac{-t}{n_r^{k_r+1}}\zeta_{n_r}^{\dia}(k_1,\dots, k_{r-1};N+t)\frac{(1-t)_{n_r-1}}{(n_r-1)!}\frac{(1-N)_{n_r}}{(1-N-t)_{n_r}},
\]
we have the conclusion from \cref{lem:induction-expansion}.
\end{proof}
The following is the main result of this section.
For an index $\bk$, we denote its Hoffman dual by $\bk^{\vee}$; see \cite[Introduction]{SekiYamamoto2020} for the definition.
\begin{theorem}\label{thm:F=G}
For an index $\bk$, we have
\[
F_N(\bk;t)=G_N(\bk^{\vee};t).
\]
\end{theorem}
\begin{remark}
Taking the limit $t\to\infty$ in this equality yields Hoffman's identity (proved by Hoffman~\cite[Theorem~4.2]{Hoffman2015}; Kawashima also independently proved this in \cite{Kawashima2009}).
\end{remark}
\begin{proof}[Proof of $\cref{thm:dia-Kawashima}$ from $\cref{thm:F=G}$]
Rewrite \cref{prop:F-harmonic} as the equation for $G_N$ using \cref{thm:F=G}, apply \cref{prop:Taylor-G}, and compare coefficients of $t^m$.
(The expression obtained directly is of the same form as \cite[Corollary~3.2]{KanekoXuYamamoto2021}; to get the expression used in \cref{thm:dia-Kawashima} (which coincides with that of \cite[Corollary~5.4]{Kawashima2009}), one rewrites it as explained in the proof of \cite[Corollary~5.4]{Kawashima2009}.)
\end{proof}
\begin{remark}
Kawashima’s original proof of (algebraic) Kawashima's relations for multiple zeta values relies on the interpolation property of the so-called Kawashima function and on the analytic properties of Newton series.
For the Kawashima function, see Yamamoto's survey article~\cite{Yamamoto2019}.
On the other hand, Kaneko--Xu--Yamamoto~\cite{KanekoXuYamamoto2021} have given an alternative proof of Kawashima's relations based on the theory of regularizations of Hurwitz-type multiple zeta values.
Since taking $N\to\infty$ in \cref{thm:dia-Kawashima} recovers Kawashima's relations for multiple zeta values, one may regard our argument as yet another new proof.
Since the interpolation property of the Kawashima function is proved via Hoffman’s identity and our expression $G_N(\bk^{\vee};t)$ is an analogue of the Newton series expression of the Kawashima function, the flavor of our argument closely resembles Kawashima’s original proof.
The difference lies in whether one interpolates multiple star harmonic sums or generalizes them directly.
Unlike Kawashima’s analytic approach, ours requires no analytic arguments and involves only finite sums.
\end{remark}
\subsection{\texorpdfstring{Proof of $\cref{thm:F=G}$}{Proof of F=G}}
We prove \cref{thm:F=G} using the connected sum method introduced by Seki--Yamamoto~\cite{SekiYamamoto2019}.
See \cite{Seki2020} for the outline of the proof based on this method and terminology.
\begin{definition}[Connectors]
For positive integers $n, m<N$, we define \emph{connectors} $C_{N,t}(n,m)$, $C_{N,t}(N-n+t,m)$, $C_{N,t}(n,N-m+t)$, and $C_{N,t}(N-n+t,N-m+t)$ by
\begin{align*}
C_{N,t}(n,m)&\coloneqq(-1)^{n-1}n\binom{m}{n}\frac{(1-N-t)_{n+m}}{(1-N-t)_n(1-N-t)_m},\\
C_{N,t}(N-n+t,m)&\coloneqq\frac{N-n+t}{N-n-m+t}C_{N,t}(n,m),\\
C_{N,t}(n,N-m+t)&\coloneqq C_{N,t}(n,m-1),\\
C_{N,t}(N-n+t,N-m+t)&\coloneqq \frac{(N-n+t)(N-m+t)}{m(N-n-m+t)}C_{N,t}(n,m).
\end{align*}
When $m=1$, $C_{N,t}(n,N-1+t)=C_{N,t}(n,0)\coloneqq0$.
\end{definition}
For an index $\bk=(k_1,\dots, k_r)$, we define $\bk_{\to}$, ${}_{\leftarrow}\bk$, $\bk_{\uparrow}$, and ${}_{\uparrow}\bk$ by
\[
\begin{aligned}
\bk_{\to} &\coloneqq (k_1, \dots, k_r, 1), 
\quad &{}_{\leftarrow}\bk &\coloneqq (1, k_1, \dots, k_r), \\
\bk_{\uparrow} &\coloneqq (k_1, \dots, k_{r-1}, k_r+1),
\quad &{}_{\uparrow}\bk &\coloneqq (k_1+1, k_2, \dots, k_r),
\end{aligned}
\]
respectively.
\begin{definition}[Connected sum]
For indices $\bk=(k_1,\dots, k_r)$ and $\bl=(l_1,\dots, l_s)$, we define the \emph{connected sums} $Z_{N,t}(\bk;\bl)$ and $Z_{N,t}(\bk_{\uparrow};\varnothing)$ by
\begin{align*}
&Z_{N,t}(\bk;\bl)\\
&\coloneqq\sum_{A\subset[r]_{\bk}^1}\sum_{B\subset[s]}\sum_{(n_1,\dots, n_r)\in S_{r,N}^{\star}(A)}\sum_{(m_1,\dots, m_s)\in\overline{S}_{s,N}(B)}\\
&\qquad(-1)^{\#B}C_{N,t}\bigl((1-\delta_{r\in A})n_r+\delta_{r\in A}(N-n_r+t),(1-\delta_{s\in B})m_1+\delta_{s\in B}(N-m_1+t)\bigr)\\
&\qquad\times\Biggl(\prod_{i\in A}\frac{1}{N-n_i+t}\Biggr)\Biggl(\prod_{i\in[r]\setminus A}\frac{1}{n_i^{k_i}}\Biggr)\Biggl(\prod_{j\in B}\frac{1}{(N-m_j+t)^{l_j}}\Biggr)\Biggl(\prod_{j\in[s]\setminus B}\frac{1}{m_j^{l_j}}\Biggr)
\end{align*}
and
\[
Z_{N,t}(\bk_{\uparrow};\varnothing)\coloneqq\sum_{A\subset[r]^1_{\bk_{\uparrow}}}\sum_{(n_1,\dots, n_r)\in S^{\star}_{r,N}(A)}C_{N,t}(n_r,N-1)\Biggl(\prod_{i\in A}\frac{1}{N-n_i+t}\Biggr)\Biggl(\prod_{i\in[r]\setminus A}\frac{1}{n_i^{k_i}}\Biggr)\frac{1}{n_r}.
\]
\end{definition}
\begin{proposition}[Boundary conditions]\label{prop:boundary}
For an index $\bk$, we have
\begin{equation}\label{eq:1st_boundary}
Z_{N,t}(\bk_{\uparrow};\varnothing)=G_N(\bk;t)
\end{equation}
and
\begin{equation}\label{eq:2nd_boundary}
Z_{N,t}((1);\bk)=F_N(\bk;t).
\end{equation}
\end{proposition}
\begin{proof}
The equality \eqref{eq:1st_boundary} follows from
\[
C_{N,t}(n_r,N-1)\frac{1}{n_r}=(-1)^{n_r-1}\binom{t}{n_r}\frac{(1-N)_{n_r}}{(1-N-t)_{n_r}}
\]
and $[r]^1_{\bk_{\uparrow}}=[r]^1_{\bk}\setminus\{r\}$.
By the definition of the connected sum, in order to prove \eqref{eq:2nd_boundary}, it suffices to check that
\[
\sum_{n=1}^mC_{N,t}(n,m)\frac{1}{n}+\sum_{n=1}^mC_{N,t}(N-n+t,m)\frac{1}{N-n+t}=1
\]
and
\[
\sum_{n=1}^{m-1}C_{N,t}(n,N-m+t)\frac{1}{n}+\sum_{n=1}^mC_{N,t}(N-n+t,N-m+t)\frac{1}{N-n+t}=1
\]
for a positive integer $m$.
These two equalities are both equivalent to
\[
\sum_{n=0}^m(-1)^n\binom{m}{n}\frac{(1-N-t)_{n+m-1}}{(1-N-t)_n}=0
\]
and this follows from
\[
\sum_{n=0}^m(-1)^n\binom{m}{n}n^a=0
\]
for any $a\in[m-1]_0$.
\end{proof}
\begin{lemma}\label{lem:1st_transport}
For positive integers $n, m<N$, we have
\begin{align}
&C_{N,t}(n,m)\frac{1}{n}=\sum_{n\leq k\leq m}C_{N,t}(n,k)\frac{1}{k}-\sum_{n<k\leq m}C_{N,t}(n,N-k+t)\frac{1}{N-k+t},\label{eq:1st1}\\
&C_{N,t}(n,N-m+t)\frac{1}{n}=\sum_{n\leq k<m}C_{N,t}(n,k)\frac{1}{k}-\sum_{n<k<m}C_{N,t}(n,N-k+t)\frac{1}{N-k+t},\label{eq:1st2}\\
&0=\sum_{n\leq k\leq m}C_{N,t}(N-n+t,k)\frac{1}{k}-\sum_{n\leq k\leq m}C_{N,t}(N-n+t,N-k+t)\frac{1}{N-k+t},\label{eq:1st3}\\
&0=\sum_{n\leq k<m}C_{N,t}(N-n+t,k)\frac{1}{k}-\sum_{n\leq k<m}C_{N,t}(N-n+t,N-k+t)\frac{1}{N-k+t}.\label{eq:1st4}
\end{align}
\end{lemma}
\begin{proof}
The first two equalities \eqref{eq:1st1} and \eqref{eq:1st2} follow from
\[
\Bigl(C_{N,t}(n,k)-C_{N,t}(n,k-1)\Bigr)\frac{1}{n}=C_{N,t}(n,k)\frac{1}{k}-C_{N,t}(n,k-1)\frac{1}{N-k+t}.
\]
The latter equalities \eqref{eq:1st3} and \eqref{eq:1st4} follow from
\[
C_{N,t}(N-n+t,k)\frac{1}{k}=C_{N,t}(N-n+t,N-k+t)\frac{1}{N-k+t}.\qedhere
\]
\end{proof}
\begin{lemma}\label{lem:2nd_transport}
For positive integers $n, m<N$, we have
\begin{align}
&\sum_{n\leq k\leq m}C_{N,t}(k,m)\frac{1}{k}+\sum_{n<k\leq m}C_{N,t}(N-k+t,m)\frac{1}{N-k+t}=C_{N,t}(n,m)\frac{1}{m},\label{eq:2nd1}\\
&\sum_{n\leq k\leq m}C_{N,t}(k,m)\frac{1}{k}+\sum_{n\leq k\leq m}C_{N,t}(N-k+t,m)\frac{1}{N-k+t}=C_{N,t}(N-n+t,m)\frac{1}{m},\label{eq:2nd2}\\
&\sum_{n\leq k< m}C_{N,t}(k,N-m+t)\frac{1}{k}+\sum_{n<k\leq m}C_{N,t}(N-k+t,N-m+t)\frac{1}{N-k+t}\label{eq:2nd3}\\
&=C_{N,t}(n,N-m+t)\frac{1}{N-m+t},\notag\\
&\sum_{n\leq k<m}C_{N,t}(k,N-m+t)\frac{1}{k}+\sum_{n\leq k\leq m}C_{N,t}(N-k+t,N-m+t)\frac{1}{N-k+t}\label{eq:2nd4}\\
&=C_{N,t}(N-n+t,N-m+t)\frac{1}{N-m+t}.\notag
\end{align}
\end{lemma}
\begin{proof}
The first equality \eqref{eq:2nd1} follows from
\[
C_{N,t}(k-1,m)\frac{1}{k-1}+C_{N,t}(N-k+t,m)\frac{1}{N-k+t}=\Bigl(C_{N,t}(k-1,m)-C_{N,t}(k,m)\Bigr)\frac{1}{m}.
\]
The second equality \eqref{eq:2nd2} is an immediate consequence of \eqref{eq:2nd1}.
The third equality \eqref{eq:2nd3} follows from
\begin{align*}
&C_{N,t}(k-1,m-1)\frac{1}{k-1}+\frac{N-m+t}{m(N-k-m+t)}C_{N,t}(k,m)\\
&=\Bigl(C_{N,t}(k-1,m-1)-C_{N,t}(k,m-1)\Bigr)\frac{1}{N-m+t}.
\end{align*}
The final equality \eqref{eq:2nd4} is an immediate consequence of \eqref{eq:2nd3}.
\end{proof}

\begin{theorem}[Transport relations]\label{thm:transport}
Let $\bk$ and $\bl$ be indices.
Then, we have
\begin{equation}\label{eq:1st_transport_Z}
Z_{N,t}(\bk_{\uparrow};\bl)=Z_{N,t}(\bk;{}_{\leftarrow}\bl)
\end{equation}
and
\begin{equation}\label{eq:2nd_transport_Z}
Z_{N,t}(\bk_{\to};\bl)=Z_{N,t}(\bk;{}_{\uparrow}\bl).
\end{equation}
For \eqref{eq:1st_transport_Z},
\[
Z_{N,t}(\bk_{\uparrow};\varnothing)=Z_{N,t}(\bk;(1))
\]
also holds.
\end{theorem}
\begin{proof}
The transport relation \eqref{eq:1st_transport_Z} follows from \cref{lem:1st_transport}.
The transport relation \eqref{eq:2nd_transport_Z} follows from \cref{lem:2nd_transport}.
\end{proof}
\cref{thm:F=G} follows from \cref{thm:transport} and the boundary conditions (\cref{prop:boundary}) in exactly the same way as in \cite[Section~3]{SekiYamamoto2020} and \cite[Section~6]{Seki2020}.

\end{document}